\documentclass[11pt]{article}

\usepackage[T1]{fontenc}
\usepackage{lmodern}

\usepackage[utf8]{inputenc}
\usepackage[hang]{footmisc}

\usepackage[english]{babel}
\usepackage{ellipsis}
\usepackage[tracking=true]{microtype}%
\DeclareMicrotypeSet*[tracking]{my}%
  { font = */*/*/sc/* }%
\SetTracking{ encoding = *, shape = sc }{ 45 }%

\usepackage{stmaryrd}
\usepackage{mathrsfs}

\usepackage[intlimits]{amsmath}
\usepackage{amsfonts}
\usepackage{amssymb}
\usepackage{amsthm}
\usepackage{enumitem}

\usepackage[margin=1in]{geometry} 
\usepackage{xcolor}
\usepackage{graphicx}
\usepackage{tikz}

\usepackage{mleftright}\mleftright
\usepackage[
	bookmarks=true,
	bookmarksnumbered=true,
	unicode=false,
	pdftoolbar=true,
	pdfmenubar=true,
	pdffitwindow=false,
	pdfstartview={FitH},
	pdftitle={},
	pdfauthor={},
	pdfsubject={},
	pdfcreator={},
	pdfproducer={},
	pdfkeywords={},
	pdfnewwindow=true,
	colorlinks=true,
	linkcolor=black,
	citecolor=black,
	filecolor=blue,
	urlcolor=black]{hyperref}

\makeatletter
	\def\@cite#1#2{[\textbf{#1}\if@tempswa , #2\fi]}
	\def\@biblabel#1{[#1]}								
\makeatother
\makeatletter
\newcommand\RedeclareMathOperator{%
  \@ifstar{\def\rmo@s{m}\rmo@redeclare}{\def\rmo@s{o}\rmo@redeclare}%
}
\newcommand\rmo@redeclare[2]{%
  \begingroup \escapechar\m@ne\xdef\@gtempa{{\string#1}}\endgroup
  \expandafter\@ifundefined\@gtempa
     {\@latex@error{\noexpand#1undefined}\@ehc}%
     \relax
  \expandafter\rmo@declmathop\rmo@s{#1}{#2}}
\newcommand\rmo@declmathop[3]{%
  \DeclareRobustCommand{#2}{\qopname\newmcodes@#1{#3}}%
}
\@onlypreamble\RedeclareMathOperator
\makeatother
\DeclareMathOperator{\R}{\mathbb{R}}

\DeclareMathOperator{\dist}{dist}
\DeclareMathOperator{\rdist}{r-dist}

\DeclareMathOperator{\conv}{conv}

\DeclareMathOperator{\ldel}{ldel}
\DeclareMathOperator{\ldiv}{ldiv}
\DeclareMathOperator{\del}{del} 
\RedeclareMathOperator{\div}{div}

\DeclareMathOperator{\Gr}{\operatorname{Gr}}

\newtheorem{theorem}{Theorem}

\newtheorem{corollary}[theorem]{Corollary}
\newtheorem{lemma}[theorem]{Lemma}

\newtheorem{example}[theorem]{Example}
\newtheorem{remark}{Remark}[section]

\newcommand{\N}{\mathbb{N}}
\newcommand{\NN}{N^{-\frac 2{n-1}}}

\newcommand{\Sp}{\mathbb{S}^{n-1}}
\renewcommand{\S}{\mathbb{S}}

\newcommand{\E}{{\mathbb{E}}}
\newcommand{\Pro}{\mathbb{P}}

\newcommand{\cP}{\mathscr{P}} 
\newcommand{\cC}{\mathscr{C}} 
\newcommand{\cK}{\mathcal{K}}

\newcommand{\dV}{\widetilde{V}}
\newcommand{\dW}{\widetilde{W}}

\renewcommand{\phi}{\varphi}

\newcommand\blfootnote[1]{%
  \begingroup
  \renewcommand\thefootnote{}\footnote{#1}%
  \addtocounter{footnote}{-1}%
  \endgroup
}

\title{Intrinsic and dual volume deviations of convex bodies and polytopes}
\author{Florian Besau, Steven Hoehner and Gil Kur}
\date{\today}

\begin{document}

\setcounter{footnote}{0}
\maketitle
\begin{abstract}
We establish estimates for the asymptotic best approximation of the Euclidean unit ball by polytopes under a notion of distance induced by the intrinsic volumes.
We also introduce a notion of distance between convex bodies that is induced by the Wills functional, and apply it to derive asymptotically sharp bounds for approximating the  ball in high dimensions. 
Remarkably, it turns out that there is a polytope which is almost optimal with respect to all intrinsic volumes simultaneously, up to  absolute constants. 

Finally, we establish asymptotic formulas for the best approximation of smooth convex bodies by polytopes with respect to a distance induced by dual volumes, which originate from Lutwak's dual Brunn--Minkowski theory.
\end{abstract}




\section{Introduction}

The approximation of convex bodies by polytopes is a classical topic in geometry with an extensive history; we  refer the interested reader to, e.g.,  the  surveys \cite{Barany:2007, Bronshtein:2008, Hug:2013} and the monograph \cite{Gruber:2007} for a proper treatment of the subject.

In this article, we focus on the asymptotic best approximation of the Euclidean unit ball $D_n\subset \R^n$. A natural question dating back to Fejes T\'oth \cite[Ch.\ 5.5]{FejesToth:1953} is: How well can we approximate the volume of the Euclidean unit ball $D_3\subset \R^3$ by an inscribed polytope with $N$ vertices?
Gruber \cite{Gruber:1988,Gruber:1991} answered this question asymptotically not only for the ball, but for all smooth convex bodies in all dimensions $n\geq 2$ (see also \cite{Boroczky:2000a} for generalizations to less smooth bodies and \cite{GMR:1994,GMR:1995} for constructions). For example, if we fix $N\in \N$ and consider the set $\cP^i_N$ of all polytopes that are contained in $D_n$ and have at most $N$ vertices,
then Gruber's result, together with asymptotic results for the dimensional constants obtained by Gordon, Reisner and Schütt \cite{GRS:1997} and improved by Mankiewicz and Schütt \cite{MankiewiczSchutt:2000, MankiewiczSchutt:2001}, imply
\begin{equation}\label{eqn:volume_best_approx_inside}
	\lim_{N\to \infty} N^{\frac{2}{n-1}} \min_{P\in \cP_N^i}  |D_n \setminus P|
	= \frac{n |D_n|}{2}\, \left(1+O\left(\frac{\ln n}{n}\right)\right).
\end{equation}
Here and throughout the paper, $|K|$ denotes the $n$-dimensional volume of a compact set $K$ in $\R^n$, and we use the asymptotic notation $O(g(n))$ to indicate that the  sequence $(f_n)_{n\in\N}$ grows at most at the rate $g(n)$ does, i.e., $\limsup_{n\to\infty} \left|f_n/g(n)\right| < \infty$. The exact value of the dimensional constant on the right-hand side of \eqref{eqn:volume_best_approx_inside} is only known for $n\in\{2,3\}$. 
Results similar to \eqref{eqn:volume_best_approx_inside} have been obtained for the mean width and volume functionals with a positive continuous weight function \cite{GlasauerGruber:1997, Ludwig:1999}. More recently, results of a similar nature were obtained for convex bodies in spaces of constant curvature, such as, for example, the unit sphere \cite{BLW:2018, Fodor:2019}.
In addition, results on volume best approximation have recently found applications in statistical machine learning theory \cite{CDSS:2018,DaganKur:2019, DKS:2017,DKS:2019, Kur:2017}.

Remarkably, as noted by Schütt and Werner \cite[Sec.\ 1.5]{SchuttWerner:2003}, when comparing \eqref{eqn:volume_best_approx_inside} with the expected volume difference of a random polytope $P_N$ generated as the convex hull of $N$ independent random points chosen uniformly from the unit sphere $\Sp := \partial D_n$, one finds
\begin{equation}\label{eqn:volume_rand_approx_inside}
	\lim_{N\to\infty} N^{\frac{2}{n-1}}\, \E |D_n \setminus P_N| = \frac{n|D_n|}{2}\, \left(1+O\left(\frac{\ln n}{n}\right)\right)
\end{equation}
(see also \cite{Affentranger:1991, BuchtaMuller:1984, Muller:1990} and Theorem \ref{affentranger:thm} below). 
Observe that the difference between the dimensional constants in the right-hand sides of \eqref{eqn:volume_best_approx_inside} and \eqref{eqn:volume_rand_approx_inside} is of the order $\frac{\ln n}{n}$ and therefore vanishes as $n\to \infty$.
Moreover, the dimensional constant in \eqref{eqn:volume_best_approx_inside} is bounded above by the dimensional constant in \eqref{eqn:volume_rand_approx_inside}, which can be calculated explicitly for all $n\geq 2$ (see \eqref{eqn:random_constants} below).
If one replaces  volume by  surface area or mean width, then similar observations can be drawn by comparing results on best approximation \cite{BoroczkyCsikos:2009, BoroczkyLudwig:1999, HSW:2018, Ludwig:1999, LSW:2006,  Schneider:1987} and on random approximation \cite{BFV:2010, BaranyThale:2017, BFH:2013, BHH:2008, BoroczkyReitzner:2004, Reitzner:2002, Reitzner:2003, SchuttWerner:2003, Thale:2018, TTW:2018}.

\medskip
In this paper, we are interested in comparing the random and best approximation of the Euclidean unit ball with respect to the intrinsic volumes, which are also known as quermassintegrals.
More precisely, the intrinsic volumes $V_j$ for $j\in \{0,1,\dotsc,n\}$ are implicitly defined for any convex body $K\subset \R^n$ via Steiner's formula 
\begin{equation*}
	|K+rD_n| = \sum_{j=0}^n r^{j} |D_{j}|\, V_{n-j}(K), \quad \forall r\geq 0.
\end{equation*}
More explicitly, by Kubota's integral formula, we have that
\begin{equation}\label{eqn:kubota}
	V_j(K) = \binom{n}{j} \frac{|D_n|}{|D_j||D_{n-j}|} \int_{\Gr_j(\R^n)} |\pi_E(K)|\, d\nu_j(E),\quad \text{for $j\in \{0,1,\dotsc,n\}$,}
\end{equation}
where $\Gr_j(\R^n)$ is the Grassmannian of all $j$-dimensional linear subspaces of $\R^n$, $\nu_j$ is the uniquely determined Haar probability measure on $\Gr_j(\R^n)$, and $\pi_E(K)$ denotes the orthogonal projection of $K$ to $E\in\Gr_j(\R^n)$.
Note that the $n$th intrinsic volume is just the volume, the $(n-1)$th intrinsic volume is half of the surface area, and the first intrinsic volume is the same as the mean width, up to a dimensional constant (see Section \ref{sec:preliminaries} for more background).
By Hadwiger's theorem, the intrinsic volumes span the space of all continuous rigid motion invariant valuations, and they are normalized such that if $K$ is $j$-dimensional, then $V_j(K)$ is exactly the $j$-dimensional Lebesgue measure of $K$.

To measure the distance between two arbitrarily positioned convex bodies $K$ and $L$, we introduce the \emph{$j$th intrinsic volume deviation} $\Delta_j(K,L)$ by setting
\begin{equation}\label{Deltajdef}
	\Delta_j(K,L) := V_j(K)+V_j(L) - 2V_j(K\cap L).
\end{equation}
By Groemer's extension theorem (see, e.g.\ \cite[Thm.\ 6.2.5]{Schneider:2014}), the $j$th intrinsic volume $V_j$ can be extended to a valuation on the lattice generated by all finite unions of convex bodies, also denoted by $V_j$. Then $V_j(K\cup L) = V_j(K)+V_j(L)-V_j(K\cap L)$ and we can write
\begin{equation*}
	\Delta_j(K,L) = V_j(K\cup L) - V_j(K\cap L).
\end{equation*}
The $j$th intrinsic volume deviation is symmetric, nonnegative and definite, i.e., $\Delta_j(K,L) = 0$ if and only if $K=L$. However, $\Delta_j$ is not a metric for $j\in\{0,1,\ldots,n-1\}$, as in general it does not satisfy the triangle inequality (see Appendix \ref{sec:not_a_metric}). It is continuous on the metric space of convex bodies in $\R^n$ that contain the origin in their interiors, equipped with the Hausdorff distance.
Note that $\Delta_n$ is the symmetric  difference metric (or Nikod\'ym distance) and $\Delta_{n-1}$ is half of the symmetric surface area deviation; see, e.g., \cite{Florian:1989}.
Furthermore, $\Delta_1$ is related to the $L^1$ metric considered in \cite{GlasauerGruber:1997, Ludwig:1999}, and in Appendix \ref{sec:width_rel} we compare these two notions of distance.

Finally, in the dual Brunn--Minkowski theory introduced by Lutwak \cite{Lutwak:1975}, we can make sense of all of the above constructions by essentially replacing the intrinsic volumes $V_j$ with the dual volumes $\dV_j$. This leads to the notion of dual volume deviations, for which we are able to establish best and random approximation estimates for general convex bodies.

\subsection*{Organization of the paper}
In the next section, we state our main results. In Section \ref{sec:preliminaries}, we collect important preliminary results before proving our main theorems in Sections \ref{proofmainthm}--\ref{sec:dual_thm}. In the final section, Section \ref{sec:stochasticwills},  we consider stochastic extensions of the Wills and dual Wills functionals. We also include an Appendix with Sections \ref{sec:not_a_metric}--\ref{sec:width_rel} where we prove results that we believe are well-known but could not find suitable literature, and also provide estimates for the hidden constants in most of our statements.

\section{Statement of principal results}
For a fixed convex body $K\subset \R^n$, $n\geq 2$, we are interested in minimizing the functional $P\mapsto \Delta_j(K,P)$, $j\in \{1,\dotsc,n\}$, where   special restrictions are imposed on the convex polytopes $P$. We therefore set
\begin{equation}\label{bestdef}
	\Delta_j(K, \cC_N) := \min_{P\in \cC_N} \Delta_j(K,P),
\end{equation}
where $\cC_N$ is a fixed subset of the class $\cP(\R^n)$ of all convex polytopes  in $\R^n$. We focus on the following classical types of polytopes:
\begin{align*}
	\mathscr P_N  &:= \{P \in \cP(\R^n): \text{$P$ has at most $N$ vertices}\},&
	\cP_N^i(K)&:= \{P \in \cP_N: P\subset K\},\\
	\cP_{(N)} &:= \{P\in \cP(\R^n): \text{$P$ has at most $N$ facets}\},&
	\cP_{(N)}^i(K)&:= \{P\in \cP_{(N)}: P\subset K\},\\
	&&
	\cP_{(N)}^o(K)&:= \{P\in \cP_{(N)}: P\supset K\}.
\end{align*}
In case of the ball, that is, $K=D_n$, we simply write $\cP_{N}^i$, $\cP_{(N)}^i$ and $\cP^o_{(N)}$. For each of these classes $\cC_N$, it follows from a compactness argument that there exists a {\it best-approximating polytope} $P\in\cC_N$ of $K$ that achieves the minimum in \eqref{bestdef}.

\subsection{Intrinsic volume approximation of the Euclidean ball}
\begin{theorem}\label{thm:1}
	Fix $n\geq 2$ and let $j\in \{1,\dotsc,n\}$. Then there exist absolute constants $c_1,\dotsc,c_5>0$ such that for all sufficiently large $N$, the following inequalities hold true:
	\begin{enumerate}
		\item[i)] for the inscribed case $\cC_N=\cP_N^i$ or the circumscribed case $\cC_N=\cP_{(N)}^o$, 
		\begin{equation*}
			c_1 j \, V_j(D_n) \leq \, N^{\frac{2}{n-1}}\, \Delta_j(D_n,\cC_N) \leq c_2 j \, V_j(D_n);
		\end{equation*}
		\item[ii)] in the case $\cC_N=\cP_{(N)}^i$,
		\begin{equation*}
			N^{\frac{2}{n-1}}\,\Delta_j(D_n,\cP_{(N)}^i) \leq c_3 j \, V_j(D_n);
		\end{equation*}
		\item[iii)] in the general position case $\cC_N=\cP_N$,
		\begin{equation*}
			N^{\frac{2}{n-1}}\,\Delta_j(D_n,\cP_N) \leq c_4 \min\left\{1, \frac{j \ln n}{n}\right\}\, V_j(D_n);
		\end{equation*}
		\item[iv)] and in the general position case $\cC_N=\cP_{(N)}$,
		\begin{equation*}
			N^{\frac{2}{n-1}}\,\Delta_j(D_n,\cP_{(N)}) \leq c_5 \, V_j(D_n), \qquad \text{if $j\geq n-c_0$,}
		\end{equation*}
			for some absolute constant $c_0\in\N$; see Remark \ref{constantsthm2}.
	\end{enumerate}
\end{theorem}
We prove Theorem \ref{thm:1} in Section \ref{proofmainthm}.
Not all results stated in this theorem are new. For more specific information on the known special cases of Theorem \ref{thm:1},  please see the discussion in Subsection \ref{comparison1} below.

\begin{remark}\label{constantsthm1}
In the proof of Theorem \ref{thm:1} i) for $\cC_N=\cP_N^i$ or $\cC_N=\cP_{(N)}^o$, we derive the  asymptotic inequalities
\begin{align}\label{eqn:asymptotic_bounds}
	\begin{split}
	\limsup_{N\to\infty} N^{\frac{2}{n-1}}\,\Delta_j(D_n,\cC_N) &\leq \frac{1}{2} \del_{n-1}|\partial D_n|^{\frac{2}{n-1}}j \, V_j(D_n)\\
	\liminf_{N\to\infty} N^{\frac{2}{n-1}}\,\Delta_j(D_n,\cC_N) &\geq \frac{1}{2} \div_{n-1}|\partial D_n|^{\frac{2}{n-1}}j \, V_j(D_n).
	\end{split}
\end{align}
Here $|\partial D_n|=n|D_n|$ denotes the surface area of $D_n$, and $\del_{n-1}$ and $\div_{n-1}$ are  positive constants that depend only on the dimension. These constants appear in the volume and mean width best approximations, see \cite{Gruber:1991} and \cite{GlasauerGruber:1997}, respectively, as well as Subsection \ref{relatedresults1}. 
In particular, \eqref{eqn:asymptotic_bounds} implies 
\begin{equation}\label{eqn:del_div_ineq}
	\div_{n-1}\leq \del_{n-1} \qquad \text{for all $n\geq 2$}.
\end{equation}
It is known that $\del_{n-1}|\partial D_n|^{\frac{2}{n-1}}=1+O(\frac{\ln n}{n})$ \cite{MankiewiczSchutt:2000, MankiewiczSchutt:2001} and $\div_{n-1}|\partial D_n|^{\frac{2}{n-1}}=1+\frac{\ln n}{n}+O(\frac{1}{n})$ \cite{HoehnerKur:2018}. Thus, the absolute constants $c_1,c_2$ in Theorem \ref{thm:1} are asymptotically
\begin{equation}\label{eqn:const_asympt}
	c_1 \sim c_2 = \frac{1}{2}+O\left(\frac{\ln n}{n}\right).
\end{equation}
Unfortunately, \eqref{eqn:asymptotic_bounds} does not imply the existence of the limit $\lim_{N\to \infty} N^{\frac{2}{n-1}}\, \Delta_j(D_n,\cC_N)$. 
To the best of our knowledge, the limit is only known to exist  for $j=1$ and $j=n$, and the other cases remain open.
\end{remark}

\begin{remark}
	The case $j=n-1$ and $\cC_N=\cP_{(N)}^o$ poses a notable exception because  
	\begin{equation*}
		\Delta_{n-1}(D_n,\cP_{(N)}^o) = \frac{n}{2} \Delta_{n}(D_n,\cP_{(N)}^o).
	\end{equation*}
	This follows from the fact that the volume of the cone generated by the origin and a facet tangent to the ball $D_n$ is exactly $1/n$ times the volume of the facet. This does not apply for general convex bodies, and in \cite{BoroczkyCsikos:2009} a limit theorem was established for the surface area best approximation of convex bodies by circumscribed polytopes with a bounded number of facets.
\end{remark}

\begin{remark}\label{constantsthm2}
Note that  in Theorem \ref{thm:1} iii), the upper bound $O\left(\frac{j\ln n}{n}\right)$ is better for lower order $j$ than the upper bound $O\left(1\right)$. On the other hand, for higher order $j$, e.g., $j=cn$ for some $c\in(0,1)$, the upper bound $O(1)$ is better than $O\left(\frac{j\ln n}{n}\right)$. From \cite[Thm.\ 2.4]{Kur:2017} and the proof of Theorem \ref{thm:1} iv) for $\cP_{(N)}$,
it follows that iv) holds for $\cP_{(N)}$ provided $N\geq n^n$. Moreover, by \cite[Rmk.\ 2.6]{Kur:2017} and the proof of  iv), the bound on the number of facets can be improved to $N\geq 10^n$, which causes a change to the value of $c_5$. For $N\geq n^n$,
the value of $c_5$ was given in \cite[Rmk.\ 2.5]{Kur:2017} in the cases $j=n$ and $j=n-1$ (see Theorem \ref{kur2} and Remark \ref{gilconstants} below). 
Using these values, $c_5$ can be estimated recursively for $j\in \{n-c_0,\dotsc,n\}$ using the argument in the proof of iv). Please see Subsection \ref{arbitraryfacetspf} for more details. 
\end{remark}

\subsection{Polytopes with a bounded number of \texorpdfstring{$k$-faces}{k-faces}}

Instead of bounding the number of vertices or facets of the approximating polytopes, one may ask for results on best approximation with respect to polytopes that have a bounded number of $k$-faces. In this setting, asymptotic bounds for the volume and mean width were previously obtained in \cite{Boroczky:2000b}. More recently, asymptotic results were obtained  for the volume \cite{BGT:2008} and Hausdorff \cite{BFVedges:2008} approximations of $C^2$ convex bodies in $\R^3$ by polytopes with a restricted number of edges.

For a polytope $P$ in $\R^n$ and an integer $k\in\{0,\ldots,n-1\}$, let $f_k(P)$ denote the number of $k$-faces of $P$.
For all $n\geq 3$ and any polytope $P\in \cP(\R^n)$, by the handshaking lemma we have
\begin{equation}\label{handshaking}
	f_1(P) \geq \frac{n}{2}f_0(P)> f_0(P) \quad \text{and}\quad f_{n-2}(P) \geq \frac{n}{2}f_{n-1}(P)> f_{n-1}(P).
\end{equation}
For a simplicial polytope $P^s$, these inequalities can be extended to
\begin{equation}\label{unimodality}
	f_0(P^s)<f_1(P^s)<\cdots<f_{\lfloor n/2\rfloor}(P^s) \quad \text{and}\quad
	f_{\lfloor 3(n-1)/4\rfloor}(P^s)>\cdots>f_{n-1}(P^s),
\end{equation}
where $\lfloor x\rfloor$ denotes the integer part of $x\in \R$.
This result is due to Bj\"orner \cite{Bjorner:1984}, who called this property of the $f$-vector $(f_0(P^s), f_1(P^s),\ldots, f_{n-1}(P^s))$ of a simplicial polytope $P^s$ in $\R^n$ its ``75\% unimodality''.

Using \eqref{handshaking} and \eqref{unimodality}, we derive an immediate corollary to Theorem \ref{thm:1}. Define the following classes of polytopes in $\R^n$:
\begin{align*}
	\cP_{k,N}^{i} &:= \{P\in \cP(\R^n): \text{$P\subset D_n$ and $f_k(P) \leq N$}\},\\
	\cP_{k,N}^{o} &:= \{P\in \cP(\R^n): \text{$P\supset D_n$ and $f_k(P) \leq N$}\}.
\end{align*}
We also  let $\cP_{k,N}^{i,s}\subset \cP_{k,N}^i$ and $\cP_{k,N}^{o,s}\subset \cP_{k,N}^o$ denote the corresponding subsets of  simplicial polytopes.
Note that $\cP_{0,N}^{i} = \cP_N^i$ and $\cP_{n-1,N}^{o} = \cP_{(N)}^o$.
\begin{corollary}\label{thm:cor1}
Let $n\geq 2$ and $j\in \{1,\dotsc,n\}$.
\begin{itemize}

\item[i)] For inscribed polytopes with a bounded number of edges, that is, $\cC_N=\cP_{1,N}^i$, 
	or for circumscribed polytopes with a bounded number of ridges, that is, $\cC_N=\cP_{n-2,N}^o$, it holds that
	\begin{equation}\label{eqn:lowerbound_edges}
		\liminf_{N\to\infty} N^{\frac{2}{n-1}}\, \Delta_j(D_n,\cC_N) 
			\geq 2^{-\frac{n+1}{n-1}} n^{\frac{2}{n-1}} \div_{n-1} |\partial D_n|^{\frac{2}{n-1}}jV_j(D_n).
\end{equation}

\item[ii)] Fix $k\in\{0,\ldots,\lfloor n/2\rfloor\}$. Then
	\begin{equation}\label{eqn:lowerbound_k_1}
		\liminf_{N\to\infty} N^{\frac{2}{n-1}}\, \Delta_j(D_n,\cP_{k,N}^{i,s}) \geq \frac{1}{2}\div_{n-1} |\partial D_n|^{\frac{2}{n-1}}jV_j(D_n).
	\end{equation}

\item[iii)] Fix $k\in\{\lfloor 3(n-1)/4\rfloor,\ldots,n-1\}$. Then
	\begin{equation}
		\liminf_{N\to\infty} N^{\frac{2}{n-1}}\, \Delta_j(D_n,\cP_{k,N}^{o,s}) \geq \frac{1}{2}\div_{n-1} |\partial D_n|^{\frac{2}{n-1}}jV_j(D_n).
	\end{equation}
\end{itemize}
\end{corollary}
\begin{proof}
By \eqref{handshaking} we have $\cP_{1,N}^i \subset \cP_{\lfloor 2N/n \rfloor}^i$. Thus, by Theorem \ref{thm:1} i) and \eqref{eqn:asymptotic_bounds} we obtain
\begin{equation*}
	\liminf_{N\to\infty} N^{\frac{2}{n-1}} \Delta_j(D_n,\cP_{1,N}^{i})  
	\geq \liminf_{N\to\infty} N^{\frac{2}{n-1}} \Delta_j(D_n,\cP_{\lfloor 2N/n\rfloor}^i)
	\geq \left(\!\frac{n}{2}\!\right)^{\frac{2}{n-1}} \frac{1}{2}\div_{n-1}|\partial D_n|^{\frac{2}{n-1}}j\, V_j(D_n).
\end{equation*}
The case $\cC_N=\cP_{n-2,N}^o$ and parts ii) and iii) follow analogously by \eqref{handshaking} and \eqref{unimodality}.
\end{proof}

Equivalently, we may formulate Corollary \ref{thm:cor1} as a  bound on the minimal number of $k$-faces of an inscribed or circumscribed simplicial polytope required to obtain an $\varepsilon$-approximation of the ball with respect to the intrinsic volume deviation.
\begin{corollary}
	Let $n\geq 2$ and $j\in\{1,\dotsc,n-1\}$.
	There exist absolute constants $c_5$, $c_6>0$ such that the following holds true for all sufficiently small $\varepsilon>0$:
	\begin{enumerate}
		\item[i)] If $P_\varepsilon\subset D_n$ is a simplicial polytope such that $V_j(D_n)-V_j(P_\varepsilon) \leq \varepsilon$, then
			\begin{equation*}
				f_k(P_\varepsilon) \geq c_5 \left(\frac{jV_j(D_n)}{2\varepsilon}\right)^{\frac{n-1}{2}} \ln n
			\end{equation*}
			for all $k\in\{0,\dotsc, \lfloor n/2\rfloor\}$. For $k\in \{0,1\}$, this bound holds true even if $P_\varepsilon$ is not simplicial.
		\item[ii)] If $P_\varepsilon \supset D_n$ is a simplicial polytope such that $V_j(P_\varepsilon) - V_j(D_n) \leq \varepsilon$, then
			\begin{equation*}
				f_k(P_\varepsilon) \geq c_6  \left(\frac{jV_j(D_n)}{2\varepsilon}\right)^{\frac{n-1}{2}} \ln n
			\end{equation*}
			for all $k\in\{\lfloor 3(n-1)/4\rfloor,\dotsc,n-1\}$. For $k\in\{n-2,n-1\}$, this bound holds true even if $P_\varepsilon$ is not simplicial.
	\end{enumerate}
\end{corollary}
\begin{proof}
	Let $P_\varepsilon\subset D_n$ be a simplicial polytope such that $V_j(D_n)-V_j(P_\varepsilon)\leq \varepsilon$.
	Set $N:=f_k(P_\varepsilon)$. If $\varepsilon$ is sufficiently small, then $N$ has to be large. Thus, by \eqref{eqn:lowerbound_k_1} we find
	\begin{equation*}
		f_k(P_\varepsilon)^{\frac{2}{n-1}}  \varepsilon \geq N^{\frac{2}{n-1}} \, \Delta_j(D_n, \cP_n^i) \geq \frac{jV_j(D_n)}{2} \div_{n-1}|\partial D_n|^{\frac{2}{n-1}}.
	\end{equation*}
	Since $\div_{n-1} |\partial D_n|^{\frac{2}{n-1}} = 1+ \frac{\ln n}{n} + O(\frac{1}{n})$, there exists an absolute constant $c>0$ such that $\div_{n-1} |\partial D_n|^{\frac{2}{n-1}} \geq 1+ \frac{c\ln n}{n}$ for all $n\geq 2$.
	Therefore,
	\begin{equation*}
		f_k(P_\varepsilon) 
		\geq \left(\frac{jV_j(D_n)}{2\varepsilon}\right)^{\frac{n-1}{2}} \left(1+\frac{c\ln n}{n}\right)^{\frac{n-1}{2}}
		\geq \left(\frac{jV_j(D_n)}{2\varepsilon}\right)^{\frac{n-1}{2}} \left(1 + \frac{c \ln n}{4}\right),
	\end{equation*}
	which establishes i). With similar arguments, ii) follows.
\end{proof}

\begin{remark}
	One may also consider simple polytopes instead of simplicial polytopes and derive corresponding results. Since the polar of a simple polytope is simplicial, the $f$-vector of simple polytope $P$ is also ``75\% unimodal'', i.e.,
\begin{equation*}
	f_0(P)<f_1(P)<\cdots<f_{\lfloor n/4 \rfloor}(P) \quad \text{and}\quad
	f_{\lfloor (n-1)/2\rfloor}(P)>\cdots>f_{n-1}(P).
\end{equation*}
\end{remark}

\subsection{Simultaneous approximation and the Wills deviation}

From the proof of Theorem \ref{thm:1} follows a remarkable  approximation property of the Euclidean ball: There is a polytope which is almost optimal for the ball with respect to all intrinsic volumes simultaneously.

\begin{corollary}\label{cor:sim_approx}
	For all sufficiently large $N$ there exist polytopes $P_N\in \cP_N^i$, respectively $P_N\in \cP_{(N)}^o$, such that
	\begin{equation*}
		c_1 n N^{-\frac{2}{n-1}} \leq \max_{j=1,\dotsc,n} \frac{\Delta_j(D_n, P_N)}{V_j(D_n)} \leq  c_2 n N^{-\frac{2}{n-1}},
	\end{equation*}
	where $c_1,c_2>0$ are the same absolute constants from Theorem \ref{thm:1}.
\end{corollary}

For another way to quantify how well a polytope approximates the ball in all intrinsic volumes simultaneously, we shall use the classical Wills functional $W:=\sum_{j=0}^n V_j$ (see, e.g., \cite{McMullen:1991, Vitale:1996}). For convex bodies $K,L\subset \R^n$, we therefore define the \emph{Wills deviation} $\Delta_\Sigma(K,L)$ by
\begin{equation*}
	\Delta_\Sigma(K,L) := W(K)+W(L)-2W(K\cap L) = \sum_{j=1}^n \Delta_j(K,L).
\end{equation*}
The Wills deviation is continuous on convex bodies that contain the origin in their interiors and is positive definite, but in general it does not satisfy the triangle inequality; see Appendix \ref{sec:not_a_metric} for specific counterexamples.  

\begin{theorem}\label{willsthm}
	Set $\widehat W(D_n):=\sum_{j=0}^n jV_j(D_n)$. Then with the same absolute constants $c_1,c_2,c_3>0$ from Theorem \ref{thm:1} and another absolute constant $c_7>0$,  for all sufficiently large $N$ the following estimates hold true:
	\begin{align*}
		\begin{array}{lrclclr}
		i)  &c_1\, \widehat W(D_n) &\!\!\!\!\leq\!\!\!\!& N^{\frac{2}{n-1}}\, \Delta_\Sigma(D_n,\cC_N)        &\!\!\!\!\leq\!\!\!\!& c_2 \widehat W(D_n)
																										& \text{for $\cC_N\in\{\cP_N^i,\cP_{(N)}^o\}$;}\\
		ii) &c_1\, \widehat W(D_n) &\!\!\!\!\leq\!\!\!\!& N^{\frac{2}{n-1}}\, \Delta_\Sigma(D_n,\cP_{k,N}^{i,s})        &           
																										&& \text{for $k\in\{0,1,\ldots,\lfloor n/2\rfloor\}$;}\\
		iii)&c_1\, \widehat W(D_n) &\!\!\!\!\leq\!\!\!\!& N^{\frac{2}{n-1}}\, \Delta_\Sigma(D_n,\cP_{k,N}^{o,s})        &           
																										&& \text{for $k\in\{\lfloor 3(n-1)/4\rfloor,\ldots,n-1\}$;}\\
		iv) &                            &&N^{\frac{2}{n-1}}\, \Delta_\Sigma(D_n,\cP_{(N)}^i) &\!\!\!\!\leq\!\!\!\!& c_3\, \widehat{W}(D_n);\\
		v)  &                            &&N^{\frac{2}{n-1}}\, \Delta_\Sigma(D_n,\cP_N)       &\!\!\!\!\leq\!\!\!\!& c_7\frac{\ln n}{n} \widehat{W}(D_n). 
		\end{array}
	\end{align*}
	Furthermore, the bound in ii)  for $k=1$ and the bound in iii) for $k=n-2$ also hold true for nonsimplicial polytopes.
\end{theorem}
We prove Theorem \ref{willsthm} in Section \ref{WillsThmSection}. 
In Section \ref{sec:stochasticwills} we consider a generalization of Theorem \ref{willsthm} (see Theorem \ref{genwillsthm}) for the stochastic Wills functional, which is an extension of the Wills functional introduced by Vitale \cite{Vitale:1996}. This extension can also be found in \cite{Hadwiger:1975} without the probabilistic notation (see also \cite{Kampf:2009}).

\begin{remark}\label{remarkWills}
	Since $V_j(D_n)=\binom{n}{j}|D_n|/|D_{n-j}|$ (see \eqref{eqn:kubota}), we can express $\widehat W(D_n)$ as
	\begin{equation*}
		\widehat{W}(D_n) = \frac{n |D_n|}{|D_{n-1}|} \sum_{j=1}^n \binom{n-1}{j-1} \frac{|D_{n-1}|}{|D_{n-j}|} = V_1(D_n) W(D_{n-1}).
	\end{equation*}
	Note that $V_1(D_n) = O(\sqrt{n})$ and $W(D_n) \leq \exp(V_1(D_n)) = O\big(\exp(\sqrt{n})\big)$ (see, e.g., \cite{McMullen:1991}).
\end{remark}

We say that a convex body $K\subset \R^n$ \emph{admits a rolling ball (from the inside)} if at every boundary point of $K$ there exists a Euclidean ball contained in $K$ of positive radius that touches $\partial K$ at $x$. As a corollary of results in \cite{BFH:2013, Reitzner:2002}, we  extend the upper bound in Theorem \ref{willsthm} i)  from the ball to all convex bodies with a rolling ball. This result is proven in Section \ref{WillsThmSection}.

\begin{theorem}\label{gilwillthm}
	Let $K\subset \R^n$ be a convex body that admits a rolling ball from the inside. Then
	\begin{equation}\label{gilwill}
		\limsup_{N\to\infty} N^{\frac{2}{n-1}}\, \Delta_\Sigma(K,\cP_N^i(K)) \leq n^{\frac{2}{n-1}} \sum_{j=1}^n
			\beta(n,j) \left(\,\int_{\partial K} H_{n-j}(K,x)^{\frac{n-1}{n+1}} H_{n-1}(K,x)^{\frac{1}{n+1}}\,d\mu_{\partial K}(x)\right)^{\frac{n+1}{n-1}},
	\end{equation}
	where $\beta(n,j)$ is defined by \eqref{eqn:beta} for $j\in\{1,\ldots,n-1\}$, $H_j(K,x)$ is the $j$th normalized elementary symmetric function of the (generalized) principal curvatures of $K$ at $x$, and $\mu_{\partial K}$ is the surface area measure of $K$, i.e., the $(n-1)$-dimensional Hausdorff measure in $\R^n$ restricted to $\partial K$.
\end{theorem}
If $K=D_n$, then the right-hand side of \eqref{gilwill} is asymptotically equal to $\left(\frac{1}{2}+O(\frac{\ln n}{n})\right) \widehat{W}(D_n)$, and therefore is asymptotically equal to the upper bound obtained in Theorem \ref{willsthm} i). Since $\Delta_\Sigma(K,\cP_N)\leq \Delta_\Sigma(K,\cP_N^i(K))$, inequality \eqref{gilwill} also gives a trivial asymptotic upper bound for the Wills deviation of $K$ and an arbitrarily positioned polytope with at most $N$ vertices.


\subsection{Dual volume approximation of convex bodies}

The classical Brunn--Minkowski theory arises from the combination of volume and the Minkowski addition of convex bodies. The dual Brunn--Minkowski theory, introduced by Lutwak \cite{Lutwak:1975, Lutwak:1988, Lutwak:1990}, originates by replacing Minkowski addition with radial addition. Many classical notions from the Brunn--Minkowski theory, such as the Brunn--Minkowski inequality, mixed volumes, surface area and curvature measures, etc., have found ``dual'' analogues in this theory. We refer to \cite{AHH:2018, Bernig:2015, BHP:2018, BLYZ:2013, HLYZ:2016, LYZ:2018, Zhao:2018} for the most recent progress, and to \cite{Gardner:2007, GVV:2003} for a succinct introduction.

The \emph{radial addition} of two convex bodies $K$ and $L$ that contain the origin in their interiors is defined by
\begin{equation*}
	K\, \widetilde{+}\, L := \{ x+y : \text{if $x\in K$ and $y\in L$ are collinear with the origin}\}.
\end{equation*}
The set $K\,\widetilde{+}\,L$ is  a star body, but in general it is not convex. Radial addition gives rise to the \emph{radial (or dual) Steiner formula}
\begin{equation}\label{eqn:radial_steiner}
	|K\,\widetilde{+}\, r D_n| = \sum_{i=0}^n r^j |D_j| \widetilde{V}_{n-j}(K),
\end{equation}
which implicitly defines the \emph{dual volumes} $\widetilde{V}_j$. For other recent generalizations of Steiner's formula,  see \cite{ParpatitsSchuster:2012, TatarkoWerner:2018}. Here we have used the normalization $\widetilde{V}_j(D_n) = V_j(D_n)$ for $\widetilde{V}_j$, which is different from the one considered in \cite{GVV:2003}.
More explicitly,  Lutwak  \cite{Lutwak:1979} established the following ``dual'' to Kubota's formula \eqref{eqn:kubota}:
\begin{equation}\label{eqn:polar_kubota}
	\widetilde{V}_j(K) = \binom{n}{j} \frac{|D_n|}{|D_j||D_{n-j}|} \int_{\Gr_j(\R^n)} |K\cap E|\, d\nu_j(E), \quad \text{for $j\in\{0,1,\dotsc,n\}$}.
\end{equation}
The dual volumes are continuous, rotation-invariant valuations on convex bodies that contain the origin in their interiors. 

We are interested in approximating the dual volume of a convex body by polytopes. Hence, we define the \emph{$j$th dual volume deviation} $\widetilde{\Delta}_j(K,L)$ between two convex bodies $K$ and $L$ that contain the origin in their interiors as
\begin{equation}
	\widetilde{\Delta}_j(K,L) := \widetilde{V}_j(K) + \widetilde{V}_j(L) - 2 \widetilde{V}_j(K\cap L).
\end{equation}
Please note that a notion of dual volume difference was previously considered in \cite{Lv:2010}.

\medskip

To state our next theorem, we need some more notation. Define the weighted curvature measure $\Omega_j$ of a convex body $K$  in $\R^n$ that contains the origin in its interior by
\begin{equation}\label{eqn:dual_curvature_measure}
	\Omega_j(K) := \int_{\partial K} \|x\|^{\frac{(j-n)(n-1)}{n+1}} H_{n-1}(K,x)^{\frac{1}{n+1}} \, d\mu_{\partial K}(x).
\end{equation}
Here $\mu_{\partial K}$ is the surface area measure on the boundary of $K$, i.e., the $(n-1)$-dimensional Hausdorff measure in $\R^n$ restricted to $\partial K$, and $H_{n-1}(K,\cdot)$ denotes the generalized Gauss--Kronecker curvature. Notice that $\Omega_j(K)$ is a weighted version of Blaschke's classical notion of affine surface area $\Omega_n(K)$, which was extended from smooth convex bodies to all convex bodies by Schütt and Werner \cite{SchuttWerner:1990}, and independently by Lutwak \cite{Lutwak:1991} (see also  \cite{Leichtweiss:1986}). The affine surface area $\Omega_n$ is equi-affine invariant, that is, $\Omega_n(AK+x) = \left|\det A\right|^{\frac{n-1}{n+1}} \Omega_n(K)$ for all $A\in \operatorname{GL}(\R^n)$ and $x\in \R^n$, whereas $\Omega_j$ is only rotation invariant for $j\neq n$, i.e., $\Omega_j(RK) = \Omega_j(K)$ for all orthogonal transformations $R$. Furthermore, since $\Omega_n$ is upper semi-continuous (see \cite{Lutwak:1991} and \cite{Ludwig:2001}), it follows that $\Omega_j$ is upper semi-continuous on all convex bodies that contain the origin in their interiors.

\begin{theorem}\label{dualthm}
	Let $K\subset \R^n$ be a convex body of class $C^2$ that contains the origin in its interior and let $j\in\{1,\dotsc,n\}$.
	Furthermore, let $\cC_N$ be either $\cP_N$, $\cP_N^i(K)$, $\cP_{(N)}$, or $\cP_{(N)}^o(K)$, and set 
	\begin{equation}\label{eqn:c_del_div}
		\gamma_{n-1} := \begin{cases}
		    	\ldel_{n-1} & \text{if }\cC_N = \cP_N,\\
		    	\del_{n-1}  & \text{if }\cC_N = \cP_N^i(K),\\
		    	\ldiv_{n-1} & \text{if }\cC_N = \cP_{(N)},\\
		    	\div_{n-1}  & \text{if }\cC_N = \cP_{(N)}^o(K).
		    \end{cases}
	\end{equation}
	Then
	\begin{equation*}
		\lim_{N\to \infty} N^{\frac{2}{n-1}}\,\widetilde{\Delta}_j(K,\cC_N) 
			= \frac{\gamma_{n-1}}{2} \frac{j V_j(D_n)}{n|D_n|}\, \Omega_j(K)^{\frac{n+1}{n-1}}.
	\end{equation*}
\end{theorem}
\noindent Please see Subsection \ref{sec:Del_Div} for more information on the dimensional constant $\gamma_{n-1}$.

\smallskip

We prove Theorem \ref{dualthm} in Section \ref{sec:dual_thm} by relating the problem to the weighted volume best approximation 
of convex bodies as considered in \cite{GlasauerGruber:1997, Ludwig:1999}.
In Theorem \ref{thm:dualthm_gen} below, we
extend this result from $j\in\{1,\dotsc,n\}$ to all $q\in \R$ by considering the natural analytic extension of the 
dual volumes from $\widetilde{V}_j$ to $\widetilde{V}_q$.

\medskip
Random approximation of convex bodies with respect to the intrinsic volumes has been considered before in \cite{Affentranger:1991, BFH:2013, Reitzner:2002}.
By an extension \cite[Satz 10.1]{Ziebarth:2014} of the random approximation results in \cite{BFH:2013} to weighted volumes (see Theorem \ref{thm:random_weighted_approx} below), we derive the following random approximation results for the dual volumes.

\begin{theorem}\label{thm:dual_optimal}
	Let $K$ be a convex body that admits a rolling ball and contains the origin in its interior. Choose $N$ points $X_1,\dotsc,X_N$  at random from $\partial K$ independently and according to the probability density function $\widetilde{\psi}_j:\partial K\to (0, \infty)$ defined by
	\begin{equation*}
		\widetilde{\psi}_j(x) := \frac{1}{\Omega_j(K)} \|x\|^{\frac{(j-n)(n-1)}{n+1}} H_{n-1}(K,x)^{\frac{1}{n+1}}
	\end{equation*}
	for all $x\in \partial K$ where $H_{n-1}(K,x)$ is defined, and set $\widetilde{\psi}_j(x)=0$ otherwise.
	\begin{enumerate}
		\item[i)] Set $P_N^{\widetilde{\psi}_j}:=\conv\{X_1,\dotsc,X_N\}$.
			Then
			\begin{equation}\label{eqn:random_weighted_approx_optimal}
				\lim_{N\to \infty} N^{\frac{2}{n-1}}\, \E\widetilde{\Delta}_j(K,P_N^{\widetilde{\psi}_j}) 
					=  \frac{\beta(n,n)}{n|D_n|} j V_j(D_n) \Omega_j(K)^{\frac{n+1}{n-1}}.
			\end{equation}
		\item[ii)] Assume further that $K$ is of class $C^2$ and set $P_{(N)}^{\widetilde{\psi}_j} := \bigcap_{i=1}^N H_i^-$, 
			where $H_i^-$ is the closed supporting halfspace of $K$ at $X_i$ that contains $K$. Then for any convex body $L$ that contains $K$ in its interior,
			\begin{equation}
				\lim_{N\to \infty} N^{\frac{2}{n-1}}\, \E\widetilde{\Delta}_j(K,P_{(N)}^{\widetilde{\psi}_j} \cap L) 
					= \frac{\beta(n,1)}{V_1(D_n)} j V_j(D_n) \Omega_j(K)^{\frac{n+1}{n-1}}.
			\end{equation}
			
	\end{enumerate}
\end{theorem}
This result follows from Theorem \ref{thm:random_weighted_approx}, Theorem \ref{thm:circumscribed} and Lemma \ref{lem:dual_to_weighted}. Please see  Subsection \ref{sec:dualWills} for more details.

\begin{remark}
	More generally, a similar limit theorem holds in Theorem \ref{thm:dual_optimal} for any continuous probability density $\varphi:\partial K \to (0, \infty)$. However, by Hölder's inequality it follows that $\widetilde{\psi}_j$ is the optimal density, i.e., the right-hand side of \eqref{eqn:random_weighted_approx_optimal} is minimal for $\widetilde{\psi}_j$ (see Subsection \ref{sec:dualWills}).
	In particular, we find that
	\begin{equation}\label{eqn:random_approx_dual}
		\lim_{N\to\infty} \frac{\E \widetilde{\Delta}_{j}(K,P_N^{\widetilde{\psi}_j})}{\widetilde{\Delta}_{j}(K,\cP_N^i(K))} 
		= 1 + O\left(\frac{\ln n}{n}\right),\qquad
		\lim_{N\to\infty} \frac{\E \widetilde{\Delta}_{j}(K,P_{(N)}^{\widetilde{\psi}_j}\cap L)}{\widetilde{\Delta}_{j}(K,\cP_{(N)}^o(K))} 
		= 1 + O\left(\frac{\ln n}{n}\right).
	\end{equation}
	Thus, also in the dual setting we see that best and random approximation are asymptotically equivalent in high dimensions. 
\end{remark}
\noindent The proof of this remark is given  in Appendix \ref{sec:asymptotic_estimates}; see  \eqref{eqn:random_approx_dual_estim} and \eqref{eqn:asymtotics_random_weighted_circum_estim}.

\medskip
We also motivate the definition of a \emph{dual Wills functional} $\dW$ for a convex body $K\subset \R^n$ that contains the origin in its interior by
\begin{equation}
	\dW(K):= \sum_{j=0}^n \dV_j(K),
\end{equation}
and we define the \emph{dual Wills deviation} by
\begin{equation}
	\widetilde\Delta_\Sigma(K,L) := \dW(K)+ \dW(L)-2 \dW(K\cap L) = \sum_{j=1}^n\widetilde{\Delta}_j(K,L)
\end{equation}
for all convex bodies $K$ and $L$ that contain the origin in their interiors.

\begin{theorem}\label{dualwillscor}
	Let $K$ be a convex body of class $C^2$ that contains the origin in its interior. Furthermore, let $\cC_N$ be either $\cP_N$, $\cP_N^i(K)$, $\cP_{(N)}$, or $\cP_{(N)}^o(K)$, and set $\gamma_{n-1}$ as in \eqref{eqn:c_del_div}. Then
	\begin{equation}\label{eqn:dualWills_lower}
		\liminf_{N\to \infty} N^{\frac{2}{n-1}} \, \widetilde{\Delta}_{\Sigma}(K,\cC_N) 
			\geq \frac{\gamma_{n-1}}{2n|D_n|} \sum_{j=1}^n j V_j(D_n)\, \Omega_j(K)^{\frac{n+1}{n-1}}.
	\end{equation}
	Moreover, if $K$ is a convex body that admits a rolling ball from the inside and contains the origin in its interior, then
	\begin{equation} \label{eqn:dualWills_upper}
		\limsup_{N\to \infty} N^{\frac{2}{n-1}} \, \widetilde{\Delta}_{\Sigma}(K,\cP_N^i(K)) 
			\leq n^{\frac{2}{n-1}}\frac{\beta(n,n)}{n|D_n|} \sum_{j=1}^n j V_j(D_n)\, \Omega_j(K)^{\frac{n+1}{n-1}},
	\end{equation}
	and if $K$ is a convex body of class $C^2$ and contains the origin in its interior, then
	\begin{equation}\label{eqn:dualWills_upper_circumscribed}
		\limsup_{N\to \infty} N^{\frac{2}{n-1}} \, \widetilde{\Delta}_{\Sigma}(K,\cP_N^o(K)) 
			\leq n^{\frac{2}{n-1}}\frac{\beta(n,1)}{V_1(D_n)} \sum_{j=1}^n j V_j(D_n)\, \Omega_j(K)^{\frac{n+1}{n-1}}.
	\end{equation}
\end{theorem}
We prove  Theorem \ref{dualwillscor} in Subsection \ref{sec:dualWills} using an extension of results in \cite{BFH:2013} for the weighted random approximation of convex bodies. Please see Theorem \ref{thm:random_weighted_approx} below for more details.

\begin{remark}
	For the inscribed case $\cC_{N}=\cP_{N}^i(K)$ we have
	\begin{equation}\label{eqn:asymtotics_random_weighted}
		\frac{\del_{n-1}}{2n|D_n|} = n^{\frac{2}{n-1}} \frac{\beta(n,n)}{n|D_n|} \left(1+O\left(\frac{\ln n}{n}\right)\right),
	\end{equation}
	and for the circumscribed case $\cC_{N} = \cP_{(N)}^o(K)$ we have
	\begin{equation}\label{eqn:asymtotics_random_weighted_circum}
		\frac{\div_{n-1}}{2n|D_n|} = n^{\frac{2}{n-1}} \frac{\beta(n,1)}{V_1(D_n)} \left(1+O\left(\frac{\ln n}{n}\right)\right).
	\end{equation}
	Therefore, the lower and upper bounds in Theorem \ref{dualwillscor} are almost equal in high dimensions; see \eqref{eqn:beta_estim_ok} and \eqref{eqn:asymtotics_random_weighted_circum_estim} in Appendix \ref{sec:asymptotic_estimates}  for the details.
\end{remark}


\subsection{Comparison with known results}\label{comparison1}

To the best of our knowledge, this paper is the  first to give estimates for intrinsic volume approximation by arbitrarily positioned polytopes. It is also the first to give asymptotically sharp lower bounds for the best inscribed and circumscribed approximations of the ball. In particular, the inscribed result shows that the random construction of Affentranger \cite{Affentranger:1991} is optimal up to a term of order $O(\frac{\ln n}{n})$.  Thus, as the dimension tends to infinity, random approximation of the ball by inscribed polytopes is as good as best approximation under the intrinsic volume deviation; see Corollary \ref{cor:AffentrangerRemark}. The main results of this paper address questions of Gruber, who asked for estimates on asymptotic best approximation of convex bodies by polytopes with respect to intrinsic volumes (\cite{Gruber:2007}, p. 216).  Most of the bounds in Theorem \ref{thm:1} were previously known only for $j\in\{1,n-1,n\}$. More specifically:

\begin{enumerate}
	\item For $j\in \{1,\dotsc,n\}$ and $\cC_N\in\{\cP_N^i(K), \cP_{(N)}^o(K)\}$, one may use the Steiner formula and results on approximation of convex bodies under the Hausdorff distance \cite{Schneider:1987} to obtain an upper bound for general convex bodies; see \cite{GlasauerGruber:1997}. In particular, by this argument one may obtain the upper bound in Theorem \ref{thm:1} i). The upper bound for  $\cC_N = \cP_N^i$ and $j\in \{1,\dotsc,n\}$ also follows from \cite[Thm.\ 5]{Affentranger:1991}.
	
	\item For $j=n$ and $\cC_N\in\{\cP_{N}, \cP_N^i(K),\cP_{(N)}, \cP_{(N)}^o(K)\}$, the limit theorems in \cite{Boroczky:2000a, Gruber:1993a, Ludwig:1999} for the symmetric volume difference of convex bodies of class $C^2$, together with estimates for the dimensional constants in \cite{Boroczky:2000b, HoehnerKur:2018, Kur:2017, LSW:2006, MankiewiczSchutt:2000, MankiewiczSchutt:2001}, imply precise bounds for the ball.
	
	\item For $j=n$ and $\cC_N=\cP_{(N)}^i$, the upper bound in Theorem \ref{thm:1} ii) was established in \cite{GMR:1994}. 
	
	\item For $j=n-1$ and $\cC_N=\cP_N$, an upper bound for convex bodies of class $C^2_+$ was  recently established in \cite{GTW:2018}. Previously, an upper bound for  $j=n-1$ and $\cC_N=\cP_N$ was given in \cite{HSW:2018} for the special case of the ball.
	
	\item For $j=n-1$ and $\cC_N=\cP_{(N)}$, a lower bound was established in \cite{HSW:2018} and an upper bound was established in \cite{Kur:2017}.  

	\item For $j=1$, the intrinsic volume deviation $\Delta_1$ is related to the $L^1$ metric $\delta_1$ defined by
	\begin{equation*}
		\delta_1(K,L) = \|h_K-h_L\|_{L^1(\S^{n-1})} = \int_{\S^{n-1}} \left|h_K(u) - h_L(u) \right|\, d\sigma(u),
	\end{equation*}
	where $\S^{n-1}=\partial D_n$ is the unit sphere, $\sigma$ is the uniform probability measure on $\S^{n-1}$ and $h_K(u) = \max\{ x\cdot u : x\in K\}$ is the support function of $K$. We have $\Delta_1(K,L) \geq V_1(D_n)\delta_1(K,L)$ with equality if and only if $K\cup L$ is convex; see Theorem \ref{thm:delta_1} in Appendix \ref{sec:width_rel}. Hence, for $\cC_N\in\{\cP_N^i(K),\cP_{(N)}^o(K)\}$, the limit theorems  in \cite{Boroczky:2000a,GlasauerGruber:1997} for the approximation of $C^2$ convex bodies by polytopes under $\delta_1$, together with the estimates for the dimensional constants in \cite{HoehnerKur:2018, Kur:2017, MankiewiczSchutt:2000,MankiewiczSchutt:2001},  
	yield precise bounds also for $\Delta_1$.

	However, for $\cC_N\in\{\cP_N, \cP_{(N)}\}$ one only has $\Delta_1\geq V_1(D_n)\delta_1$ in general, and therefore we have to distinguish between the two notions. For $\delta_1$, limit theorems were obtained by Ludwig \cite{Ludwig:1999} and estimates for the dimensional constants can be found in \cite{Kur:2017, LSW:2006}. By a result of Eggleston \cite{Eggleston:1957}, it follows that in the plane $\R^2$ we have $\Delta_1(D_2,\cP_N) = \Delta_1(D_2,\cP_N^i)$, i.e., a polygon with $N$ vertices is best-approximating for the unit disk if and only if it is inscribed. In particular, this yields $\Delta_1(D_2,\cP_N) > V_1(D_n) \delta_1(D_2,\cP_N)$ for all $N\geq 3$. Very recently, Fodor \cite{Fodor:2019} proved an analogue of Eggleston's result for the hyperbolic plane $\mathbb{H}^2$, and showed that it fails on the sphere $\S^2$.	
\end{enumerate}

To the best of our knowledge, polytopal approximation of convex bodies with respect to the Wills functional and stochastic Wills functional have not been considered before, and our asymptotic bounds in the inscribed and circumscribed cases are optimal, up to absolute constants. Furthermore, approximation with respect to  dual volumes also appears to be new and, as we show, is strongly tied to best and random approximation of convex bodies with respect to weighted volumes as considered in \cite{BFH:2013, GlasauerGruber:1997, Ludwig:1999, Reitzner:2002, SchuttWerner:2003}.


\section{Preliminaries}\label{sec:preliminaries}

As a general reference on convex bodies, we refer to the monographs \cite{Gardner:2006, Gruber:2007, Schneider:2014}. In the following, we  collect the necessary notions and classical results on best and random approximation needed in our proofs.

\paragraph{Notation and Definitions}

We shall work in $n$-dimensional Euclidean space $\R^n$, $n\geq 2$, equipped with inner product $ x\cdot y=\sum_{i=1}^n x_i y_i$ and Euclidean norm $\|x\|=\sqrt{x \cdot x}$. 
For $n\in\N$, we set $[n]:=\{1,\ldots,n\}$. 

A \emph{convex body} is a convex and compact subset of $\R^n$ with nonempty interior.  
We write $\mathcal{K}(\R^n)$ for the space of convex bodies in $\R^n$ endowed with the Hausdorff metric, and $\cK_0(\R^n)$ denotes the subspace of convex bodies in $\R^n$ that contain the origin in their interiors. 
The boundary of a convex body $K$ is denoted $\partial K$, and $\mu_{\partial K}$ denotes the surface area measure of $K$. 
The $n$-dimensional volume of $K$ is $|K|$, and its surface area is denoted $|\partial K| = \mu_{\partial K}(\partial K)$.

The Euclidean unit ball is the set $D_n=\{x\in\R^n: \|x\|\leq 1\}$. Its boundary $\partial D_n$ is the unit sphere $\Sp$, and $\sigma$ denotes the uniform probability measure on $\Sp$, i.e., $\sigma = \frac{\mu_{\partial D_n}}{|\partial D_n|}$. 
The volume of the Euclidean unit ball is $|D_n| = \pi^{\frac{n}{2}}/\Gamma(1+\frac{n}{2})$, and $|\partial D_n|=n|D_n|$. The following asymptotic estimate is used frequently (see Appendix \ref{sec:asymptotic_estimates} \eqref{eqn:partial_Dn}):
\begin{equation}\label{eqn:asymptotic_Dn}
	|\partial D_n|^{\frac{2}{n-1}} = \frac{2\pi e}{n} \left(1+O\left(\frac{1}{n}\right)\right).
\end{equation}
The support function  $h_K:\Sp\to\R$ of $K\in\cK(\R^n)$ is defined by $h_K(u)=\max_{x\in K} x\cdot u$, where $u\in\Sp$. The polar body $K^\circ$ of $K$ is defined by $K^\circ=\{y\in\R^n : \max_{x\in K} x\cdot y \leq 1\}$.


\subsection{Background on intrinsic volumes}\label{backgroundIVs}

The intrinsic volumes $V_1(K),\ldots,V_n(K)$ of a convex body $K\in\cK(\R^n)$ are completely determined by Steiner's formula (see, e.g.,\ \cite[Eqn.\ (4.1)]{Schneider:2014})
\begin{equation}\label{eqn:steiner}
	|K+r D_n| = \sum_{j=0}^n r^{j} |D_j| V_{n-j}(K), \quad \forall r\geq 0.
\end{equation}
Note that $K+r D_n$ is the set of all points with distance at most $r>0$ from $K$, i.e.,
	$K+r D_n =\{x\in\R^n: \dist(x,K)\leq r\}$
where $\dist(x,K)=\min\{\|x-y\|: y\in K\}$. In particular, 
	$V_j(D_n) = \binom{n}{j} \frac{|D_n|}{|D_{n-j}|}$.

The intrinsic volumes are monotonic, i.e., if $K\subset L$ then $V_j(K) \leq V_j(L)$. Thus, $\Delta_j (K,L) \geq 0$ for all $K,L\in\cK(\R^n)$. The first intrinsic volume $V_1(K)$ of $K$ is called the \emph{intrinsic width} of $K$, and Kubota's integral formula \eqref{eqn:kubota} yields
\begin{equation}
	V_1(K) = V_1(D_n) \int_{\Sp}h_K(u)\,d\sigma(u).
\end{equation}

The famous \emph{Alexandrov--Fenchel inequalities} for mixed volumes imply the following well-known inequalities.
\begin{theorem}[Isoperimetric inequalities for intrinsic volumes] \label{extended_isoperimetric}
	Let $K\in\cK(\R^n)$.
	\begin{enumerate}
		\item[i)] Extended isoperimetric inequality: For any $j\in[n]$,
			\begin{equation}\label{eqn:extended_isoperimetric}
				\left(\frac{|K|}{|D_n|}\right)^{\frac{1}{n}}
				\leq \cdots
				\leq \left(\frac{V_j(K)}{V_j(D_n)}\right)^{\frac{1}{j}} 
				\leq \cdots
				\leq \frac{V_1(K)}{V_1(D_n)}.
			\end{equation}
			Equality holds for any one of the inequalities, and then throughout all of them, if and only if $K$ is a Euclidean ball, 
			i.e., $K=r D_n+x$ for some  $r> 0$ and $x\in\R^n$.
			
		\item[ii)] The sequence $\{V_j(K)/V_j(D_n)\}_{j=0}^n$ is log-concave, i.e., for all $j\in[n-1]$,
			\begin{equation}\label{eqn:Alexandrov-Fenchel}
				\left(\frac{V_j(K)}{V_j(D_n)}\right)^2 \geq \frac{V_{j-1}(K)}{V_{j-1}(D_n)} \cdot \frac{V_{j+1}(K)}{V_{j+1}(D_n)}.
			\end{equation}
	\end{enumerate}
\end{theorem}
Equality holds in \eqref{eqn:Alexandrov-Fenchel} if $K=r D_n+x$ for some $r\geq 0$ and $x\in \R^n$, but there is no complete characterization of the equality cases of \eqref{eqn:Alexandrov-Fenchel}.
For more background on the Alexandrov--Fenchel inequalities and their numerous consequences, see, e.g., \cite[Ch.\ 7]{Schneider:2014}.

In \eqref{eqn:extended_isoperimetric} and \eqref{eqn:Alexandrov-Fenchel} one may replace the intrinsic volume $V_j$ with some other renormalization. In particular, the classical isoperimetric and Urysohn inequalities  are special cases of \eqref{eqn:extended_isoperimetric}: 
\begin{equation*}
	\frac{|\partial K|}{|\partial D_n|} = \frac{V_{n-1}(K)}{V_{n-1}(D_n)} \geq \left(\frac{|K|}{|D_n|}\right)^{\frac{n-1}{n}}, \quad 
	\frac{V_1(K)}{V_1(D_n)} \geq \left(\frac{|K|}{|D_n|}\right)^{\frac{1}{n}}.
\end{equation*}


\subsection{Best approximation of convex bodies by polytopes}\label{relatedresults1}

First, let us briefly remark on the well-posedness of our best-approximation problems. 
By the definition of $\Delta_j$, we have $\Delta_j(K,\cC_N) \geq \Delta_j(K,\cC_{N+1})$, where $\cC_N$ denotes a fixed class of polytopes from $\{\cP_N, \cP_N^i(K), \cP_{(N)}, \cP_{(N)}^o(K)\}$. Since polytopes are dense in the space of all convex bodies with respect to the Hausdorff metric and since $\Delta_j$ is positive definite, we conclude that $\Delta_j(K,\cC_N)$ monotonically decreases to $0$ as $N\to \infty$. In particular, this implies that there are sequences $(P_N)_{N\in\mathbb{N}}$ such that $P_N\in\cC_N$, $\Delta_j(K,\cC_N)\leq \Delta_j(K,P_N)$ and $\Delta_j(K,P_N)\to 0$ as $N\to \infty$. 

To prove the existence of minimizers in $\cC_N$, one may  argue as follows.
Let $N\geq n+1$ and let $x_1,\dotsc,x_N\in \R^n$ be arbitrary. The support function of a convex polytope $P_N:=\conv\{x_1,\dotsc,x_N\}\in \cP_N$ is
\begin{equation*}
 	h_{P_N}(u) = \max \{ x_i\cdot u : 1\leq i\leq N\},\quad \forall u\in \S^{n-1}.
\end{equation*}
The mapping $(x_1,\dotsc,x_N)\mapsto h_{P_N}$ is continuous with respect to the $L^\infty$-norm on continuous functions on $\S^{n-1}$, and therefore the mapping $(x_1,\dotsc,x_N)\mapsto P_N$ is continuous with respect to the Hausdorff metric.
Hence, if we draw points $x_i$ from a  convex compact subset $K\subset \R^n$, then the continuity of the functional $(x_1,\dotsc,x_N) \mapsto \Delta_j(K,P_N)$ yields the existence of a best-approximating polytope $P^b_N\in \cP_N^i(K)$ such that $\Delta_j(K,P^b_N) = \Delta_j(K,\cP_N^i(K))$. More generally, one can show that if $K$ contains a closed ball of radius $r$ in its interior and is contained in an open ball of radius $R$, i.e., $rD_n+x \subset \mathrm{int}\,K$ and $K\subset \mathrm{int}\, RD_n+y$ for some $x,y\in \R^n$, and if $\Delta_j(K,P)$ is  small, then necessarily $rD_n+x\subset P \subset RD_n+y$. Hence, again by compactness and continuity, if $N$ is large enough  there exists a best-approximating polytope $P^b_N\in \cC_N$ such that $\Delta_j(K,P^b_N) = \Delta_j(K,\cC_N)$ for $\cC_N\in\{\cP_N, \cP_N^i(K), \cP_{(N)}, \cP_{(N)}^o(K)\}$.

\smallskip
In the proof of Theorem \ref{thm:1}, we apply a result of Gruber \cite{Gruber:1993b} for the Euclidean unit ball. Fejes T\'oth \cite{FejesToth:1953} stated a version of this theorem in the plane, which was later proven by McClure and Vitale \cite{McClureVitale:1975}.
\begin{theorem}[{\cite[Thm.\ 1]{Gruber:1993b}}]\label{grubervol} 
	Fix $n\geq 2$ and let $K\subset \R^n$ be a convex body of class $C^2$.
	Then
	\begin{align*}
		\lim_{N\to\infty} N^{\frac{2}{n-1}}\, \Delta_n(K,\cP^i_N) = \frac{1}{2} \del_{n-1} \Omega_n(K)^{\frac{n+1}{n-1}},\\
		\lim_{N\to\infty} N^{\frac{2}{n-1}}\, \Delta_n(K,\cP^o_{(N)}) = \frac{1}{2} \div_{n-1} \Omega_n(K)^{\frac{n+1}{n-1}},
	\end{align*}
	where $\Omega_n(K)$ is the affine surface area of $K$ defined in \eqref{eqn:dual_curvature_measure}.
\end{theorem}
Gruber proved this theorem for convex bodies of class $C^2_+$ (i.e., $C^2$ convex bodies with everywhere positive Gauss--Kronecker curvature), which was subsequently weakened to $C^2$ by Böröczky \cite{Boroczky:2000a}.
Observe that asymptotically, the best approximation is determined by the affine surface area of $K$ and the constants $\del_{n-1}$ and $\div_{n-1}$ which depend only on the dimension $n$. We briefly recall estimates for these constants in the next section.
For the ball, we have $\Omega_n(D_n) = |\partial D_n|= n|D_n|$ and therefore
\begin{align}
	\lim_{N\to\infty} N^{\frac{2}{n-1}}\, \Delta_n(D_n,\cP^i_N) 
		&= \frac{n|D_n|}{2}\del_{n-1}|\partial D_n|^{\frac{2}{n-1}}, \label{eqn:gruberinvert}\\
	\lim_{N\to\infty} N^{\frac{2}{n-1}}\, \Delta_n(D_n,\cP^o_{(N)})
		&= \frac{n|D_n|}{2}\div_{n-1}|\partial D_n|^{\frac{2}{n-1}}.\label{eqn:gruberoutfacets}
\end{align}

Glasauer and Gruber \cite{GlasauerGruber:1997} obtained similar limit theorems for convex bodies of class $C^2$ and the metric $\Delta_1$. We only state their results for the unit ball.
\begin{theorem}[{Corollary to \cite[Thm.\ 1]{GlasauerGruber:1997}}]\label{thm:glasgrubmw}
The following asymptotic formulas hold true:
	\begin{align}
		\lim_{N\to\infty} N^{\frac{2}{n-1}}\, \Delta_1(D_n,\cP_N^i)     &= \frac{V_1(D_n)}{2}\div_{n-1}|\partial D_n|^{\frac{2}{n-1}},\\
		\lim_{N\to\infty} N^{\frac{2}{n-1}}\, \Delta_1(D_n,\cP_{(N)}^o) &= \frac{V_1(D_n)}{2}\del_{n-1}|\partial D_n|^{\frac{2}{n-1}}.
	\end{align}
\end{theorem}

\begin{remark}
	Please note that Glasauer and Gruber actually state Theorem \ref{thm:glasgrubmw} in terms of the $L^1$-metric $\delta_1(K,L)=\int_{\Sp} |h_K(u)-h_L(u)|\, d\sigma(u)$. In the case of inscribed and circumscribed polytopes, $\delta_1$ and $\Delta_1$ agree, up to a dimensional constant. In general, $\Delta_1(K,L)\geq V_1(D_n) \delta_1(K,L)$ with equality if and only if $K\cup L$ is convex. For more information, see Appendix \ref{sec:width_rel}. 
\end{remark}

\subsection{Delone triangulation and Dirichlet--Voronoi tiling numbers}\label{sec:Del_Div}

The constants $\del_{n-1}$ and $\div_{n-1}$ are called the Delone triangulation and Dirichlet--Voronoi tiling numbers in $\R^{n-1}$, respectively. They were introduced by Gruber in \cite{Gruber:1993b} and are named after Delone triangulations and Dirichlet--Voronoi tilings in $\R^{n-1}$, which are dual tessellations of $\R^{n-1}$ that arise in the proofs of the asymptotic formulas \eqref{eqn:gruberinvert} and \eqref{eqn:gruberoutfacets} in \cite{Gruber:1993b}. The exact values of these constants are known explicitly only for $n=2$ and $n=3$. 
Fejes T\'oth \cite{FejesToth:1953}  derived the values $\del_1=1/6$ and $\div_1=1/12$. The values $\del_2=\frac{1}{2\sqrt{3}}$ and $\div_2=\frac{5}{18\sqrt{3}}$ were later determined by Gruber in \cite{Gruber:1988} and \cite{Gruber:1991}, respectively.
For $n\geq 4$,  the exact values of $\div_{n-1}$ and $\del_{n-1}$ are unknown, but their asymptotic behavior has been estimated quite precisely. The best-known asymptotic estimates for $\del_{n-1}$ and $\div_{n-1}$ are:
\begin{equation}\label{eqn:del_div}
	\del_{n-1} = \frac{n}{2\pi e} \left(1 + O\left(\frac{\ln n}{n}\right)\right), \quad
	\div_{n-1} = \frac{n}{2\pi e} \left(1+\frac{\ln n}{n}+O\left(\frac{1}{n}\right)\right). 
\end{equation}
The estimate for $\del_{n-1}$ is due to Mankiewicz and Sch\"utt \cite{MankiewiczSchutt:2000, MankiewiczSchutt:2001}, and the estimate for $\div_{n-1}$ is due to Hoehner and Kur \cite{HoehnerKur:2018}. 

\paragraph{Laguerre--Delone triangulation and Laguerre--Dirichlet--Voronoi tiling numbers.}
The dimensional constants $\ldel_{n-1}$ and $\ldiv_{n-1}$ that appear in Theorem \ref{dualthm} are called the Laguerre--Delone and Laguerre--Dirichlet--Voronoi tiling numbers in $\R^{n-1}$, respectively. They were introduced by Ludwig \cite{Ludwig:1999} and are connected with Laguerre--Delone and Laguerre--Dirichlet--Voronoi tilings in $\R^{n-1}$, respectively. These tilings arise in the proofs of the asymptotic formulas \eqref{eqn:weighted_approx} for arbitrarily positioned polytopes in \cite{Ludwig:1999}. For $n=2$, it is known that $\ldel_1=\ldiv_1=1/16$ (see, e.g., \cite{Ludwig:1999}). For $n=3$, Gruber \cite{Gruber:1988} conjectured that $\ldel_2=\frac{1}{6\sqrt{3}} -\frac{1}{8\pi}$. B\"or\"oczky and Ludwig \cite{BoroczkyLudwig:1999} later proved that this is the correct value, and they also established that $\ldiv_2=\frac{5}{18\sqrt{3}} - \frac{1}{4\pi}$. For $n\geq 4$, the exact values of $\ldel_{n-1}$ and $\ldiv_{n-1}$ are again unknown. It has been determined that there exist positive absolute constants $C_1,C_2,C_3,C_4$ such that
\begin{equation}\label{eqn:ldel_ldiv}
	\frac{C_1}{n} \leq \ldel_{n-1} \leq C_2 \quad \text{ and } \quad
	C_3 \leq \ldiv_{n-1} \leq C_4.
\end{equation}
The lower and upper estimates for $\ldel_{n-1}$ are due to B\"or\"oczky \cite{Boroczky:2000b} and Ludwig, Sch\"utt and Werner \cite{LSW:2006}, respectively, and the lower and upper estimates for $\ldiv_{n-1}$ are due to Ludwig, Sch\"utt and Werner \cite{LSW:2006} and Kur \cite{Kur:2017}, respectively. In fact, Kur \cite{Kur:2017} gave the estimates
\begin{equation}
	\frac{0.25}{\pi e}+o(1) \leq \ldiv_{n-1} \leq \frac{0.97}{\pi e}+o(1).
\end{equation}
Closing the dimensional gap between the upper and lower estimates for $\ldel_{n-1}$ appears to be a difficult open problem.


\subsection{Intrinsic volume approximation of convex bodies by random polytopes}\label{sec:AffentrangerRandom}

Random constructions have frequently been used to generate well-approximating polytopes.  Remarkably, it turns out that in many cases, as the dimension tends to infinity, random approximation of smooth convex bodies is asymptotically as good as best approximation. This phenomenon has been observed in, e.g., the  volume approximation by inscribed polytopes \cite{Gruber:1993a, SchuttWerner:2003};  volume, surface area, and mean width approximation by circumscribed polytopes \cite{Boroczky:2004, BoroczkyCsikos:2009, GlasauerGruber:1997, Gruber:1993a}; and volume \cite{HoehnerKur:2018, Kur:2017, LSW:2006} and surface area \cite{HSW:2018, Kur:2017} approximation by arbitrarily positioned polytopes. 

\smallskip
Affentranger \cite{Affentranger:1991} proved the following asymptotic formula for the approximation of the ball by inscribed random polytopes under the intrinsic volume difference.

\begin{theorem}[{\cite[Thm.\ 5]{Affentranger:1991}}]\label{affentranger:thm}
 	Choose $N$ points $X_1,\ldots,X_N$ independently with respect to the uniform probability measure $\sigma$ on the unit sphere $\Sp$, and set $P_{N}:=\conv\{X_1,\ldots,X_N\}$. Then
	\begin{equation}\label{eqn:affentrangerball1}
		\lim_{N\to\infty} N^{\frac{2}{n-1}}\,\E \Delta_j(D_n, P_{N})
		= \frac{jV_j(D_n)}{2} \alpha(n,j),
	\end{equation}
	where
	\begin{equation}\label{eqn:random_constants}
		\alpha(n,j) 
		:=  \left(1-\frac{2}{n+1}\right) \, \left(\frac{n|D_n|}{|D_{n-1}|}\right)^{\frac{2}{n-1}}\,
			\frac{\Gamma\left(j+1+\frac{2}{n-1}\right)}{\Gamma(j+1)}.
	\end{equation}
\end{theorem}

In  Appendix \ref{sec:asymptotic_estimates} we verify that for all $j\in[n]$,
\begin{equation}\label{eqn:affentrager_const}
	\alpha(n,j) = 1+O\left(\frac{\ln n}{n}\right)\quad \text{as $n\to \infty$}.
\end{equation}
Thus, the constant from Theorem \ref{thm:1} i) for $\cC_N = \cP_N^i$ agrees with the constant in the right-hand side of \eqref{eqn:affentrangerball1}, up to an error term of order $O(\frac{\ln n}{n})$. This is summarized in the following corollary, where our estimates can be found in  Appendix \ref{sec:asymptotic_estimates}.
\begin{corollary}\label{cor:AffentrangerRemark}
	Choose $N$ points $X_1,\ldots,X_N$ independently with respect to the uniform probability measure $\sigma$ on the unit sphere $\Sp$, and let $P_{N}:=\conv\{X_1,\ldots,X_N\}$. Then for every $j\in[n]$,
	\begin{equation*}
		\limsup_{N\to\infty}\frac{\E \Delta_j(D_n,P_{N})}{\Delta_j(D_n,\cP_N^i)} = 1+O\left(\frac{\ln n}{n}\right).
	\end{equation*}
\end{corollary}
More generally, Reitzner \cite{Reitzner:2002} obtained an asymptotic formula for the expected $j$th intrinsic volume difference for $C_2^+$ convex bodies, which was extended in \cite{BFH:2013} to convex bodies that admit a rolling ball from the inside.
\begin{theorem}[{\cite[Thm.\ 1]{Reitzner:2002} and \cite[Thm.\ 1.2]{BFH:2013}}]\label{reitznerthm}
	Let $K$ be a convex body that admits a rolling ball from the inside. Choose $N$ points $X_1,\ldots,X_N$  at random from $\partial K$ independently according to a positive, continuous probability density function $\varphi:\partial K \to (0,\infty)$, and let $P_{N}^\varphi:=\conv\{X_1,\ldots,X_N\}$. Then
	\begin{equation}\label{reitznerrandom1}
		\lim_{N\to\infty} N^{\frac{2}{n-1}}\, \E \Delta_j(K, P_{N}^\varphi) 
		= \beta(n,j) \int_{\partial K} \varphi(x)^{-\frac{2}{n-1}} H_{n-j}(K,x) H_{n-1}(K,x)^{\frac{1}{n-1}}\, d\mu_{\partial K}(x),
	\end{equation}
	where $\beta(n,j)$ is a positive constant that depends only on $n$ and $j$, and for $k\in\{0,\ldots,n-1\}$, $H_k(K,x)$ is the $k$th normalized elementary symmetric function of the (generalized) principal curvatures of $K$ at $x$.
\end{theorem}
Putting $K=D_n$ and comparing \eqref{reitznerrandom1} with \eqref{eqn:affentrangerball1} yields 
\begin{equation}\label{eqn:beta}
	\beta(n,j) = \alpha(n,j) \frac{j V_j(D_n)}{2n|D_n|} |\partial D_n|^{-\frac{2}{n-1}} = \frac{jV_j(D_n)}{4\pi e |D_n|} \left(1+O\left(\frac{\ln n}{n}\right)\right),
\end{equation}
and in particular,
\begin{equation*}
	\beta(n,n) = \frac{1}{2} \left(1-\frac{2}{n+1}\right) |D_{n-1}|^{-\frac{2}{n-1}} \frac{\Gamma\left(n+1+\frac{2}{n-1}\right)}{\Gamma(n+1)}.
\end{equation*}

Using Hölder's inequality, Reitzner showed that the right-hand side of \eqref{reitznerrandom1} is minimized by the probability density
\begin{equation}\label{eqn:bestdensity}
	\varphi_{j}(x) := \frac{H_{n-j}(K,x)^{\frac{n-1}{n+1}}H_{n-1}(K,x)^{\frac{1}{n+1}}}{\int_{\partial K}H_{n-j}(K,x)^{\frac{n-1}{n+1}} H_{n-1}(K,x)^{\frac{1}{n+1}}\, d\mu_{\partial K}(x)}, \quad\quad \forall x\in\partial K.
\end{equation}
Choosing this density in \eqref{reitznerrandom1} yields
\begin{equation}\label{reitznerbest1}
	\lim_{N\to\infty} N^{\frac{2}{n-1}}\, \E \Delta_j(K,P_{N}^{\varphi_{j}}) 
	=  \beta(n,j) \left(\,\int_{\partial K} H_{n-j}(K,x)^{\frac{n-1}{n+1}} H_{n-1}(K,x)^{\frac{1}{n+1}}\,d\mu_{\partial K}(x)\right)^{\frac{n+1}{n-1}}.
\end{equation}

Let $\psi:\R^n\to(0,\infty)$ be a continuous function. As a generalization of the volume difference, one defines the \emph{$\psi$-weighted volume difference} by
\begin{equation}\label{eqn:weighted_volume_difference}
	\Delta_n^\psi(K,L) := \int_{K\triangle L} \psi(x) \, dx.
\end{equation}
Best approximation with respect to the $\psi$-weighted volume difference was considered in \cite{Ludwig:1999} (see Theorem \ref{thm:weighted_approx} below). We need the following generalization of the random approximation for the volume difference to the weighted volume difference, which was obtained in \cite{Ziebarth:2014}. This result is an extension of \cite[Thm.\ 1.1]{BFH:2013} for the case $j=n$ (see also \cite{SchuttWerner:2003}).

\begin{theorem}[{\cite[Satz 10.1]{Ziebarth:2014}, weighted volume extension of \cite[Thm.\ 1.1]{BFH:2013}}] \label{thm:random_weighted_approx}
	Let $K$ be a convex body that admits a rolling ball from the inside, and let $\psi:K\to (0,\infty)$ be a  weight function on $K$ that is continuous near the boundary of $K$ and such that $\sup_{K}\psi < \infty$. Choose $N$ points $X_1,\dotsc,X_N$  at random from $\partial K$ independently according to a continuous probability density function $\varphi:\partial K \to (0, \infty)$, and let $P_N^{\varphi}:=\conv\{X_1,\dotsc,X_N\}$. Then
	\begin{equation}\label{eqn:random_weighted_approx}
		\lim_{N\to \infty} N^{\frac{2}{n-1}}\, \E \Delta_n^\psi(K,P_N^\varphi) = \beta(n,n) \int_{\partial K} 
			\varphi(x)^{-\frac{2}{n-1}} \psi(x) H_{n-1}(K,x)^{\frac{1}{n-1}} \, d\mu_{\partial K}(x).
	\end{equation}
\end{theorem}

Using Hölder's inequality, we derive that given $\psi$, the minimal value of the right-hand side of \eqref{eqn:random_weighted_approx} is achieved for the probability density
\begin{equation}\label{eqn:optimal_psi}
	\widetilde{\psi}(x)
	:= \frac{\psi(x)^{\frac{n-1}{n+1}}H_{n-1}(K,x)^{\frac{1}{n+1}}}{\int_{\partial K} \psi(x)^{\frac{n-1}{n+1}} H_{n-1}(K,x)^{\frac{1}{n+1}} \, d\mu_{\partial K}(x)}, \quad\quad \forall x\in\partial K.
\end{equation}
Choosing this density in \eqref{eqn:random_weighted_approx} yields
\begin{equation}\label{eqn:random_weighted_approx_best}
		\lim_{N\to \infty} N^{\frac{2}{n-1}}\, \E \Delta_n^\psi(K,P_N^{\widetilde{\psi}}) 
		= \beta(n,n) \left(\,\,\int_{\partial K} \psi(x)^{\frac{n-1}{n+1}} H_{n-1}(K,x)^{\frac{1}{n+1}} \, d\mu_{\partial K}(x)\right)^{\frac{n+1}{n-1}}.
\end{equation}

\medskip
A dual random construction that generates random polytopes which are circumscribed around a convex body was considered in \cite{BoroczkyReitzner:2004}. Choose $N$ points $X_1,\dotsc,X_N$ randomly and independently from the boundary of a convex body $K$ of class $C^2$, and consider the random polyhedral set $P_{(N)}$ that is the intersection of all the closed halfspaces $H^-_i$ of $K$, where $H^-_i$ is the uniquely determined supporting halfspace of $K$ at $X_i\in \partial K$.

\begin{theorem}[{\cite[Thm.\ 1]{BoroczkyReitzner:2004}}]\label{thm:circumscribed}
	Let $K\subset \R^n$ be a convex body of class $C^2$ and let $L$ be an arbitrary convex body that contains $K$ in its interior. Choose $N$ points $X_1,\dotsc,X_N$  at random from $\partial K$ independently according to a  continuous probability density function $\varphi: \partial K \to (0,\infty)$, and set $P_{(N)}^{\phi}:= \bigcap_{i=1}^N H_i^-$. 
		If $\psi:L\to (0,\infty)$ is a continuous and bounded weight function, then 
		\begin{equation}\label{eqn:random_approx_circum_weighted}
			\lim_{N\to\infty} N^{\frac{2}{n-1}}\, \E \Delta_n^\psi(K,P_{(N)}^{\varphi}\cap L)
			= \beta(n,1) |D_{n-1}| \int_{\partial K}
				\varphi(x)^{-\frac{2}{n-1}} \psi(x) H_{n-1}(K,x)^{\frac{1}{n-1}} \, d\mu_{\partial K}(x).
		\end{equation}
		The optimal density is given by $\phi=\widetilde{\psi}$ as defined in \eqref{eqn:optimal_psi}, and in this case
		\begin{equation}
			\lim_{N\to\infty} N^{\frac{2}{n-1}}\, \E \Delta_n^\psi(K,P_{(N)}^{\widetilde{\psi}}\cap L)
				= \beta(n,1) |D_{n-1}| \left(\,\,\int_{\partial K} \psi(x)^{\frac{n-1}{n+1}} H_{n-1}(K,x)^{\frac{1}{n+1}} \, d\mu_{\partial K}(x)\right)^{\frac{n+1}{n-1}}.
		\end{equation}
\end{theorem}
Note that results for $\Delta_{n-1}$ and $\Delta_1$ were also obtained in \cite{BoroczkyReitzner:2004}.

\begin{remark}
	The curvature conditions on $K$ in Theorem \ref{thm:circumscribed} were recently weakened in \cite[Satz 10.4]{Ziebarth:2014}, where one only requires that $K$ slides freely inside a ball, which is equivalent to the property that $K^\circ$ admits a rolling ball from the inside (where we may assume without loss of generality that $K$ contains the origin in its interior).
	However, since in this case $K$ may have singular points, i.e., there might be more than one support hyperplane  at a fixed boundary point, one has to consider a probability distribution on the set of all hyperplanes that envelop the convex body $K$ instead of a probability distribution on $\partial K$. 
	
	In fact, as was observed in \cite{Ziebarth:2014}, the random polytope $P_{(N)}$ is in distribution equivalent to the polar of a random polytope $P_N = \conv \{Y_1,\dotsc, Y_N\}$ with vertices $Y_1,\dotsc, Y_N$ chosen at random from the boundary of $K^\circ$ with respect to a  distribution determined by the  distribution of the halfspaces that generate $P_{(N)}$.
\end{remark}


\section{Intrinsic volume approximation of the ball}\label{proofmainthm}

In this section we prove Theorem \ref{thm:1}.

\subsection{Proof of Theorem \ref{thm:1} i): Inscribed and circumscribed case}
	Let $P\subset D_n$.
	By Theorem \ref{extended_isoperimetric} i), 
	\begin{equation}\label{eqn:inscribed_ineq}
		1-\left(1-\frac{V_1(D_n)-V_1(P)}{V_1(D_n)}\right)^j \leq \frac{V_j(D_n)-V_j(P)}{V_j(D_n)} \leq 1-\left(1-\frac{V_n(D_n)-V_n(P)}{V_n(D_n)}\right)^{\frac{j}{n}}.
	\end{equation}
	Taking the minimum over all $P$ in $\cP_N^i$ we obtain
	\begin{align*}
		1- \left(1-\frac{\Delta_1(D_n,\cP_N^i)}{V_1(D_n)}\right)^j 
		\leq \frac{\Delta_j(D_n,\cP_N^i)}{V_j(D_n)}
		\leq 1- \left(1-\frac{\Delta_n(D_n,\cP_N^i)}{V_n(D_n)}\right)^{\frac{j}{n}}.
	\end{align*}
	Thus, using Theorem \ref{grubervol} and Theorem \ref{thm:glasgrubmw}, as well as the formulas $\lim_{x\to 0^+} \frac{1-(1-cx)^\alpha}{x} = c\alpha$ for $\alpha>0$ and $V_1(D_n) = \frac{n| D_n|}{|D_{n-1}|}$, we derive that 
	\begin{align}
		\limsup_{N\to\infty} N^{\frac{2}{n-1}}\, \Delta_j(D_n,\cP_N^i) &\leq \frac{1}{2} \del_{n-1}|\partial D_n|^{\frac{2}{n-1}} jV_j(D_n),\label{invertupper}\\
		\liminf_{N\to\infty} N^{\frac{2}{n-1}}\, \Delta_j(D_n,\cP_N^i) &\geq \frac{1}{2} \div_{n-1}|\partial D_n|^{\frac{2}{n-1}}j V_j(D_n)\label{invertlower}.
	\end{align}

	Analogously, if $P\supset D_n$ then
	\begin{equation}\label{eqn:circum_ineq}
		\left(1+\frac{V_n(P)-V_n(D_n)}{V_n(D_n)}\right)^{\frac{j}{n}} - 1 \leq \frac{V_j(P)-V_j(D_n)}{V_j(D_n)} \leq \left(1+\frac{V_1(P)-V_1(D_n)}{V_1(D_n)}\right)^j - 1.
	\end{equation}
	Now taking  the minimum over all $P$ in $\cP_{(N)}^o$ we deduce
	\begin{align*}
		\left(1+\frac{\Delta_n(D_n,\cP_{(N)}^o)}{V_n(D_n)}\right)^{\frac{j}{n}} - 1
		\leq \frac{\Delta_j(D_n,\cP_{(N)}^o)}{V_j(D_n)}
		\leq \left(1+\frac{\Delta_1(D_n,\cP_{(N)}^o)}{V_1(D_n)}\right)^{j} - 1.
	\end{align*}
	Hence, again by Theorem \ref{grubervol} and Theorem \ref{thm:glasgrubmw} we conclude
	\begin{align}
		\limsup_{N\to\infty} N^{\frac{2}{n-1}}\, \Delta_j(D_n,\cP_{(N)}^o)&\leq \frac{1}{2} \del_{n-1}|\partial D_n|^{\frac{2}{n-1}}jV_j(D_n)\label{outfacetsupper},\\
		\liminf_{N\to\infty} N^{\frac{2}{n-1}}\, \Delta_j(D_n,\cP_{(N)}^o)&\geq \frac{1}{2} \div_{n-1}|\partial D_n|^{\frac{2}{n-1}} jV_j(D_n)\label{outfacetslower}.
	\end{align}
\qed
\subsection{Proof of Theorem \ref{thm:1} ii): Inscribed case with bounded number of facets}

As remarked in \cite{GMR:1994}, it was proved in \cite{BronshteinIvanov:1975} that there is an absolute constant $C_1>0$
such that for all sufficiently large $N$, there exists a polytope $P^b\in \cP_{(N)}^i$ such that
\begin{equation}\label{eqn:inside_facets_bound_polytope}
	|P^b| \geq (1-C_1 n N^{-\frac{2}{n-1}})|D_n|.
\end{equation}
Thus, by Theorem \ref{extended_isoperimetric} i), for any $j\in[n]$ and all sufficiently large $N$ it holds true that
\begin{equation*}
	V_j(P^b) \geq V_j(D_n)\left(\frac{|P^b|}{|D_n|}\right)^\frac{j}{n} \geq V_j(D_n)\left(1-C_1 n N^{-\frac{2}{n-1}}\right)^{\frac{j}{n}}.
\end{equation*}
Hence, for all sufficiently large $N$,
\begin{equation*}
	\Delta_j (D_n,\cP_{(N)}^i) \leq V_j(D_n) - V_j(P^b) \leq V_j(D_n) \left(1-\left(1-C_1 n N^{-\frac{2}{n-1}}\right)^\frac{j}{n}\right),
\end{equation*}
and therefore
\begin{equation}\label{eqn:inside_facets_bound}
	\limsup_{N\to \infty} N^{\frac{2}{n-1}}\, \Delta_j (D_n,\cP_{(N)}^i) \leq C_1 jV_j(D_n).
\end{equation}
This proves Theorem \ref{thm:1} ii). \qed


\subsection{Proof of Theorem \ref{thm:1} iii): General position and bounded number of vertices}

Let $P_N:=\conv\{X_1,\dotsc,X_N\}$ be a random polytope generated by $N$ random vertices $X_1,\dotsc,X_N$ chosen independently and uniformly from $\partial D_n$.
Set
\begin{equation*}
	r=r(n,j,N) := \left(\frac{\E V_j(P_N)}{V_j(D_n)}\right)^{-\frac{1}{j}} \quad\text{ and }\quad
	R=R(n,N) := \left(\frac{\E |P_N|}{|D_n|}\right)^{-\frac{1}{n}}
\end{equation*}
so that $\E V_j(rP_N) = V_j(D_n)$ and $\E |RP_N| = |D_n|$.
In the proof of \cite[Thm.\ 1]{LSW:2006} (more specifically, \cite[Eq.\ (9) and (28)]{LSW:2006}), it was shown that there exists an absolute constant $C$ such that for all sufficiently large $N$,
\begin{equation*}
	\E\Delta_n((1-c_{n,N})D_n , P_N) \leq C |D_n| N^{-\frac{2}{n-1}}
\end{equation*}
where $c_{n,N}$ is determined by
\begin{equation*}
	(1-c_{n,N})^n |D_n| = \E |P_N|.
\end{equation*}
(See \cite[Eq.\ (10)]{LSW:2006} and  \cite[Eq.\ (3.14)]{GroteWerner:2017} for similar results.) 
Hence $R=(1-c_{n,N})^{-1}$, and since $R\to 1$ as $N\to \infty$ we derive that
\begin{equation}\label{eqn:upper_bound_vertices_LSW}
	\E \Delta_n(D_n, RP_N) \leq C R^n |D_n|  N^{-\frac{2}{n-1}} \leq C_0 |D_n|  N^{-\frac{2}{n-1}}
\end{equation}
for some absolute constant $C_0$ when $N$ is large enough.
We aim to show that there exists an absolute constant $c_4$ such that
\begin{equation}\label{eqn:upper_bound_vertices}
	\limsup_{N\to \infty} N^{\frac{2}{n-1}}\, \E\Delta_j(D_n, rP_N) \leq c_4 \min \left\{ 1 , \frac{j\ln n}{n}\right\} V_j(D_n),
\end{equation}
which will yield the desired upper bound since $\Delta_j(D_n,\cP_N)\leq \E\Delta_j(D_n,rP_N)$.

Let $\varepsilon>0$ be arbitrary.
By the aforementioned result of Affentranger in Theorem \ref{affentranger:thm}, there exists $N_0\in \N$ such that for all $N\geq N_0$,
\begin{align*}
	\E V_j(P_N) &\leq V_j(D_n) \left(1- \frac{1-\varepsilon}{2} j \alpha(n,j) N^{-\frac{2}{n-1}}\right)
	\intertext{and}
	\E |P_N| & \geq |D_n| \left(1-\frac{1+\varepsilon}{2} n \alpha(n,n) N^{-\frac{2}{n-1}} \right).
\end{align*}
This yields
\begin{equation*}
	\limsup_{N\to \infty} N^{\frac{2}{n-1}} \, (R-r) \leq \frac{1+\varepsilon}{2} \alpha(n,n) - \frac{1-\varepsilon}{2} \alpha(n,j).
\end{equation*}
Similarly, we obtain
\begin{equation*}
	\liminf_{N\to \infty} N^{\frac{2}{n-1}} \, (R-r) \geq \frac{1-\varepsilon}{2} \alpha(n,n) - \frac{1+\varepsilon}{2} \alpha(n,j).
\end{equation*}
We have $1 \leq \alpha(n,j) \leq C_1$ and
\begin{equation*}
	1+\frac{2}{n^2} \leq \frac{\alpha(n,n)}{\alpha(n,j)} \leq 1 + \min\left\{ \frac{1}{j} , 3\frac{\ln n}{n}\right\}, \quad \forall n\geq 2,\; \forall j\in[n-1], 
\end{equation*}
which is proven in Appendix \ref{sec:asymptotic_estimates}.
Since $\varepsilon>0$ was chosen arbitrarily, this yields
\begin{equation*}
	\limsup_{N\to \infty} N^{\frac{2}{n-1}} \, (R-r) 
	\leq \frac{\alpha(n,j)}{2} \left(\frac{\alpha(n,n)}{\alpha(n,j)}-1\right)
	\leq  \frac{3C_1}{2} \min\left\{\frac{1}{j}, \frac{\ln n}{n}\right\},
\end{equation*}
and $\liminf_{N\to \infty} N^{\frac{2}{n-1}} \, (R-r) > 0$.
Thus, there exists $N_1\in \N$ such that for all $N\geq N_1$ 
\begin{equation}\label{eqn:R_bounds}
	0 \leq R-r \leq 2C_1 \min\left\{\frac{1}{j}, \frac{\ln n}{n}\right\} N^{-\frac{2}{n-1}}.
\end{equation}
Since $R\geq r$, we deduce that
\begin{equation*}
	D_n \cap RP_N \subset (1+R-r) \left( D_n \cap rP_N\right),
\end{equation*}
which yields
\begin{equation*}
	|D_n\cap RP_N| - |D_n\cap rP_N| \leq \left((1+R-r)^n-1\right) |D_n\cap rP_N|
\end{equation*}
 by monotonicity. Moreover,
\begin{equation*}
	(1+R-r)^n \leq 1+ C_2 n (R-r) \quad \text{and} \quad R^n-r^n \leq C_3 n (R-r)
\end{equation*}
for some absolute constants $C_2,C_3>0$ and all sufficiently large $N$.
Thus, when $N$ is large enough,
\begin{align*}
	|\Delta_n(D_n,RP_N)-\Delta_n(D_n,rP_N)| 
	&\leq |RP_N|-|rP_N| + 2 \left(|D_n\cap RP_N| - |D_n\cap rP_N|\right)\\
	&\leq (R^n-r^n) |P_N| + 2\left((1+R-r)^n-1\right) |D_n\cap P_N|\\
	&\leq (2C_2+C_3) n (R-r) |D_n|\\
	&\leq \underbrace{2C_1(2C_2+C_3)}_{=:C_4} \min\left\{\frac{n}{j}, \ln n \right\} |D_n| N^{-\frac{2}{n-1}}.
\end{align*}
Since $\Delta_n(D_n,rP_N) \geq |D_n| - |D_n\cap rP_N|$, this yields that for all $N\geq N_1$,
\begin{equation*}
	|D_n \cap rP_N| \geq |D_n| \left( 1 - C_4\min\left\{\frac{n}{j}, \ln n \right\} N^{-\frac{2}{n-1}} - \frac{\Delta_n(D_n,RP_N)}{|D_n|}\right)_+
\end{equation*}where   $(x)_+ := \max\{x,0\}$ for $x\in\R$. 
By Theorem \ref{extended_isoperimetric} i) we derive
\begin{align}
	\E \Delta_j(D_n, rP_N) 
	&= V_j(D_n) + \E V_j(rP_N) - 2 \E V_j(D_n\cap rP_N) \notag\\
	&\leq 2V_j(D_n) - 2 V_j(D_n) \E\left(\frac{|D_n\cap rP_N|}{|D_n|}\right)^{\frac{j}{n}} \notag\\
	&\leq 2V_j(D_n) - 2 V_j(D_n) \E \left(1-C_4\min\left\{\frac{n}{j}, \ln n \right\} N^{-\frac{2}{n-1}} - \frac{\Delta_n(D_n,RP_N)}{|D_n|}\right)_+^{\frac{j}{n}}.
	\label{eqn:upper_bound_vertices_2}
\end{align}
To continue, we would like to apply a Bernoulli-type inequality, but to do so we need to verify that
\begin{equation}\label{eqn:bernoulli_bound}
	Z_N:=C_4\min\left\{\frac{n}{j}, \ln n \right\} N^{-\frac{2}{n-1}} + \frac{\Delta_n(D_n,RP_N)}{|D_n|} < 0.8
\end{equation}
holds true with high probability if $N$ is large. Consider the event $A_N:=\{P_N:|P_N| > \frac{1}{2} |D_n|\}$.
If $P_N\in A_N$ and if $R^n<5/4$, then
\begin{equation*}
	\Delta_n(D_n, RP_N) 
	\leq |RD_n\setminus P_N| 
	= R^n|D_n| - |P_N|
	< \frac{3}{4} |D_n|.
\end{equation*}
Since $R\to 1$ as $N\to \infty$, this shows that \eqref{eqn:bernoulli_bound} holds true for the event $A_N$ for sufficiently large $N$.
Furthermore, using Chebyshev's inequality,  a variance bound obtained by Reitzner \cite[Thm.\ 8]{Reitzner:2003} and Theorem \ref{affentranger:thm} for $j=n$, we bound the probability of the complementary event $A_N^c$ by
\begin{equation}\label{eqn:upper_bound_event}
	\Pro(A_N^c)
	\leq \Pro\left(\Big||P_N| - \E|P_N|\Big| \geq \E|P_N|-\frac{1}{2}|D_n|\right) 
	\leq \frac{\operatorname{Var}(|P_N|)}{(\E|P_N|-\frac{1}{2}|D_n|)^2} \leq c(n) N^{-1-\frac{4}{n-1}}
\end{equation}
for some positive constant $c(n)$ that only depends on the dimension $n$.
The trivial upper bound
\begin{equation*}
	X_N := \left(1-Z_N\right)_+^{\frac{j}{n}} \leq 1
\end{equation*}
and $X_N\geq 0$ together yield 
\begin{equation*}
    \E X_N  \geq \E[X_N|A_N]\Pro(A_N) = \E[X_N|A_N]-\E[X_N|A_N]\Pro(A_N^c) \geq \E[X_N|A_N]-\Pro(A_N^c).
\end{equation*}
Next, we apply the Bernoulli-type inequality $(1-z)^t \geq 1-3tz$ for $z\in[0,0.8)$ and $t>0$ to derive that under the event $A_N$, for all sufficiently large $N$ any realization of the random variable $X_N$ satisfies
\begin{equation*}
    X_N \geq 1-\frac{3j}{n} Z_N = 1-3C_4\min\left\{1, \frac{j\ln n}{n} \right\} N^{-\frac{2}{n-1}} -\frac{3j}{n} \frac{\Delta_n(D_n,RP_N)}{|D_n|}.
\end{equation*}
Therefore, by the previous two inequalities, as well as \eqref{eqn:upper_bound_vertices_LSW}, \eqref{eqn:upper_bound_vertices_2} and \eqref{eqn:upper_bound_event}, we finally  obtain that for large enough $N$,
\begin{align*}
	\E \Delta_j(D_n, rP_N) &\leq 2V_j(D_n)(1-\E X_N)
	\leq 2V_j(D_n)(1-\E[X_N|A_N]+\Pro(A_N^c))\\
	&\leq 2 V_j(D_n) \left( 3C_4 \min \left\{1, \frac{j\ln n}{n}\right\} N^{-\frac{2}{n-1}} + \frac{3j}{n} \frac{\E \Delta_n(D_n,RP_N)}{|D_n|} + c(n) N^{-1-\frac{4}{n-1}}\right)\\
	&\leq 6 \left(C_4+C_0+ c(n) N^{-1-\frac{2}{n-1}} \right) \min \left\{1, \frac{j\ln n}{n}\right\}  V_j(D_n) N^{-\frac{2}{n-1}}.
\end{align*}
Thus, \eqref{eqn:upper_bound_vertices} holds true. \qed

\subsection{\texorpdfstring{Proof of Theorem \ref{thm:1} iv): General position and bounded number of facets}{Proof of Theorem 1 iv)}}\label{arbitraryfacetspf}

We shall show that there is an absolute constant $C$ and an absolute constant $c_0\in[n-1]$ such that for all sufficiently large $N$, there exists a polytope $P_{(N)}\in\cP_{(N)}$ such that
\begin{equation}\label{recursive}
	\Delta_j(D_n,P_{(N)}) \leq C V_j(D_n)\NN, \quad \forall j\in\{n-c_0,\ldots,n\}.
\end{equation}
We will use a recursive argument to prove \eqref{recursive}. The main idea is to use Theorem \ref{extended_isoperimetric} ii) and bounds on $\Delta_j(D_n,P_{(N)})$ and $\Delta_{j+1}(D_n,P_{(N)})$ to derive an upper bound on $\Delta_{j-1}(D_n,P_{(N)})$, iterating from $j=n-1$ to $j=n-c_0+1$. More specifically, we show that in the $k$th step of the recursion, $\Delta_{n-k-1}(D_n,P_{(N)}) \leq C_k\Delta_{n-k}(D_n,P_{(N)})$ for some absolute constant $C_k$. The hypothesis $j\geq n-c_0$ is necessary because the constants $C_k$ blow up fast. To prove the existence of the polytope $P_{(N)}$, we use the following result of Kur \cite{Kur:2017} to initialize the recursion.

\begin{theorem}[Remark 2.5 in \cite{Kur:2017}]\label{kur2}
There exists an absolute constant $C_0$ such that for every $n \geq 2$ and $N \geq n^n$ there exists a polytope $P_{(N)}$ in $\R^n$ with at most $N$ facets which satisfies both of the following inequalities simultaneously:
\begin{equation}\label{kurSA}
	\Delta_n(D_n,P_{(N)}) \leq C_0|D_n|\NN, \qquad \Delta_{n-1}(D_n, P_{(N)}) \leq 4C_0 V_{n-1}(D_n) \NN.
\end{equation}
\end{theorem}

\begin{remark}\label{gilconstants}
It was shown in \cite{Kur:2017} that
\begin{align*}
	C_0 &= \int_0^1 t^{-1}\left(1-e^{-t\ln 2}\right)\,dt + \int_0^\infty e^{-e^t\ln2}\,dt+O\left(\frac{1}{\sqrt{n}}\right) \\
	&=\gamma+\ln\ln 2 - 2 \mathrm{Ei}(-\ln 2) + O\left(\frac{1}{\sqrt{n}}\right)
	=0.9680448\ldots+O\left(\frac{1}{\sqrt{n}}\right)
\end{align*}
where $\mathrm{Ei}(x)=-\int_{-x}^\infty t^{-1}e^{-t}\,dt$ is the exponential integral and $\gamma=0.5772\ldots$ is the Euler--Mascheroni constant. It was also shown in \cite{Kur:2017} that the inequalities \eqref{kurSA} hold simultaneously provided $N\geq 10^n$, but with a different constant $C_0$.
\end{remark}


\paragraph{\texorpdfstring{Proof of Theorem \ref{thm:1} iv)}{Proof of Theorem 1 iv)}}

We argue by induction on $j\in[n-1]$ to show that for all $N\geq \max\{n^n, (10\cdot 4^j C_0)^{\frac{n-1}{2}}\}$, the polytope $P_{(N)}$ from Theorem \ref{kur2} satisfies
\begin{equation}\label{eqn:induction_hyp}
	\Delta_{n-k}(D_n, P_{(N)}) \leq (9\cdot 2^{k-1}-5) C_0  V_{n-k}(D_n)\NN, \qquad \text{for all $k\in \{0,\dotsc,j\}$.}
\end{equation}
Since $\Delta_{n-j}(D_n,\cP_{(N)}) \leq \Delta_{n-j}(D_n,\cP_{(N)}^o)$, this only improves the upper bound for $\cP_{(N)}$ if $j$ is so small that $(9\cdot 2^{j-1}-5)C_0\leq c_2 n$. Hence in the upper bound in Theorem \ref{thm:1} iv), we restrict $j\in \{n-c_0,\dotsc,n\}$ for some absolute constant $c_0\in[n-1]$.

The statement \eqref{eqn:induction_hyp} is true for $j=1$ by Theorem \ref{kur2}. So let $j\geq 1$, and we will show that the statement \eqref{eqn:induction_hyp} holds true for $j+1$.
Inequality \eqref{kurSA} implies that for all $N\geq n^n$,
	\begin{equation}\label{inter1}
		|D_n \cap P_{(N)}| 
		=|D_n\cup P_{(N)}|-|D_n\triangle P_{(N)}| \geq|D_n|-|D_n\triangle P_{(N)}| 
		\geq (1-C_0\NN)|D_n|. 
	\end{equation}
	Hence, for all $N\geq n^n$,
 	\begin{equation}\label{volumeupper1}
		|P_{(N)}| =2|D_n \cap P_{(N)}|-|D_n|+|D_n \triangle P_{(N)}| 
		\geq 2|D_n \cap P_{(N)}| - |D_n|
		\geq (1-2C_0\NN)|D_n|.
	\end{equation}
By Theorem \ref{extended_isoperimetric} i), \eqref{inter1} and Bernoulli's inequality we obtain
\begin{equation}\label{lowerint}
	V_{n-j-1}(D_n\cap P_{(N)}) 
	\geq V_{n-j-1}(D_n) \left(\frac{|D_n\cap P_{(N)}|}{|D_n|}\right)^{\frac{n-j-1}{n}}
	\geq \left(1 - C_0\NN\right) V_{n-j-1}(D_n)
\end{equation}
for all $N\geq \max\{n^n, C_0^{\frac{n-1}{2}}\}$.
Analogously, by \eqref{volumeupper1} we derive
\begin{equation*}
	V_{n-j+1}(P_{(N)}) 
	\geq V_{n-j+1}(D_n)\left(\frac{|P_{(N)}|}{|D_n|}\right)^{\frac{n-j+1}{n}} 
	\geq \left(1-2C_0\NN\right)V_{n-j+1}(D_n)
\end{equation*}
for all $N\geq \max\{n^n, (2C_0)^\frac{n-1}{2}\}$.
The monotonicity of $V_{n-j}$ and the induction hypothesis \eqref{eqn:induction_hyp} imply
\begin{equation}\label{ru2}
	V_{n-j}(P_{(N)})
 	\leq V_{n-j}(D_n) + \Delta_{n-j}(D_n, P_{(N)})
	\leq \left(1+ (9\cdot 2^{j-1}-5) \NN\right) V_{n-j}(D_n)
\end{equation}
for all $N$ large enough.
Moreover, by Theorem \ref{extended_isoperimetric} ii) we have
\begin{align}\label{upperunion}
	V_{n-j-1}(P_{(N)}) 
	&\leq \frac{V_{n-j-1}(D_n)V_{n-j+1}(D_n)}{V_{n-j}(D_n)^2}\cdot\frac{V_{n-j}(P_{(N)})^2}{V_{n-j+1}(P_{(N)})}\nonumber\\
	&\leq \frac{\left(1+(9\cdot 2^{j-1}-5)C_0 \NN\right)^2}{\left(1-2C_0\NN\right)} V_{n-j-1}(D_n) \nonumber \\
	&\leq \left(1+(9\cdot 2^j -7) C_0 \NN\right) V_{n-j-1}(D_n)
\end{align}
for all $N\geq \max\{n^n, (10\cdot4^{j+1}C_0)^\frac{n-1}{2}\}$. 
Finally, by \eqref{lowerint} and \eqref{upperunion} we derive
\begin{align*}
	\Delta_{n-j-1}(D_n, P_{(N)}) &= V_{n-j-1}(P_{(N)})+V_{n-j-1}(D_n) - 2V_{n-j-1}(D_n\cap P_{(N)})\\
	& \leq  (9\cdot 2^j -5)C_0 V_{n-j-1}(D_n)\NN
\end{align*}
for all $N\geq \max\{n^n, (10\cdot 4^{j+1} C_0)^{\frac{n-1}{2}}\}$. Hence \eqref{eqn:induction_hyp} holds true for $j+1$.
\qed

\section{Simultaneous approximation and the Wills functional}\label{WillsThmSection}

In this section we prove Corollary \ref{cor:sim_approx}, Theorem \ref{willsthm} and Theorem \ref{gilwillthm}.

\subsection{Proof of Corollary \ref{cor:sim_approx}: Simultaneous approximation of the Euclidean ball}

By \eqref{eqn:inscribed_ineq}, we derive that for any polytope $P\subset D_n$, 
\begin{equation*}
	 \max_{j\in[n]} \left\{1-\left(1-\frac{\Delta_1(D_n,P)}{V_1(D_n)}\right)^j \right\}
	 \leq \max_{j\in[n]} \frac{\Delta_j(D_n,P)}{V_j(D_n)} 
	 \leq \max_{j\in[n]} \left\{ 1-\left(1-\frac{\Delta_n(D_n,P)}{V_n(D_n)}\right)^{\frac{j}{n}}\right\}. 
\end{equation*}
Taking the minimum over $P\in \cP_N^i$ and the limit as $N\to \infty$, we conclude
\begin{align*}
	\limsup_{N\to \infty} N^{\frac{2}{n-1}} \min_{P\in \cP_N^i} \max_{j\in[n]} \frac{\Delta_j(D_n,P)}{V_j(D_n)} \leq \frac{n}{2} \del_{n-1} |\partial D_n|^{\frac{2}{n-1}},\\
	\liminf_{N\to \infty} N^{\frac{2}{n-1}} \min_{P\in \cP_N^i} \max_{j\in[n]} \frac{\Delta_j(D_n,P)}{V_j(D_n)} \geq \frac{n}{2} \div_{n-1} |\partial D_n|^{\frac{2}{n-1}}.
\end{align*}
Hence the corollary follows for the case $\cP_N^i$. Analogously, we derive the case $\cP_{(N)}^o$ from \eqref{eqn:circum_ineq}.
\qed

\subsection{Proof of Theorem \ref{willsthm}: Bounds for the Wills deviation for the Euclidean ball}

\paragraph{Proof of i).} The lower bound follows directly from Theorem \ref{thm:1} i) and \eqref{eqn:asymptotic_bounds} since
\begin{equation*}
	\liminf_{N\to\infty} N^{\frac{2}{n-1}} \, \Delta_\Sigma(D_n,\cC_N) 
	\geq \sum_{j=0}^n\liminf_{N\to\infty} N^{\frac{2}{n-1}} \, \Delta_j(D_n,\cC_N)
	\geq \frac{1}{2}\div_{n-1}|\partial D_n|^{\frac{2}{n-1}}\widehat W(D_n).
\end{equation*}
For the upper bound, we only prove the case $\cC_N=\cP_N^i$ as the case $\cC_N=\cP_{(N)}^o$ follows similarly by performing the obvious modifications. Let $\varepsilon>0$ be arbitrary. By Theorem \ref{grubervol}, there exists $N_0\in\N$ such that for all $N\geq N_0$ there exists $P^b_N\in \cP_N^i$ that satisfies
\begin{equation*}
	|P^b_N| \geq |D_n| \left(1- (1+\varepsilon) \frac{n}{2} \del_{n-1} |\partial D_n|^{\frac{2}{n-1}} N^{-\frac{2}{n-1}}\right).
\end{equation*}
Then by Theorem \ref{extended_isoperimetric} i), for all $N\geq N_0$ we have
\begin{equation*}
	V_j(D_n) - V_j(P^b_N) 
	\leq V_j(D_n) \left(1- \left( 1- (1+\varepsilon) \frac{n}{2} \del_{n-1} |\partial D_n|^{\frac{2}{n-1}} N^{-\frac{2}{n-1}}\right)^{\frac{j}{n}} \right). 
\end{equation*}
This implies that for all $N\geq N_0$,
\begin{equation*}
	\limsup_{N\to\infty} N^{\frac{2}{n-1}} \, \Delta_\Sigma(D_n,\cP_n^i) 
	\leq \sum_{j=0}^n \limsup_{N\to\infty} N^{\frac{2}{n-1}} \, \Delta_j(D_n, P^b_N)
	\leq \frac{1+\varepsilon}{2}\del_{n-1}|\partial D_n|^{\frac{2}{n-1}}\widehat W(D_n).
\end{equation*}
Since $\varepsilon>0$ was arbitrary, the result follows. \qed

\paragraph{Proof of ii) and iii).} Analogously to i), Theorem \ref{willsthm} ii) and iii) follow from Corollary \ref{thm:cor1}.\qed

\paragraph{Proof of iv).} Analogously to i), Theorem \ref{willsthm} iv) follows from summing  inequality \eqref{eqn:inside_facets_bound_polytope} over $j\in[n]$. \qed

\paragraph{Proof of v).} 
In the proof of Theorem \ref{thm:1} iii), specifically \eqref{eqn:upper_bound_vertices}, it was shown that when $N$ is sufficiently large, the random polytope $P_N\in A_N$ satisfies
\begin{equation*}
	\Delta_j(D_n,r_jP_N) \leq c_4 \frac{j\ln n}{n} V_j(D_n) N^{-\frac{2}{n-1}}, \qquad \forall j\in[n],
\end{equation*}
where $r_j:= \left(\frac{\E V_j(P_N)}{V_j(D_n)}\right)^{-\frac{1}{j}}$. Recall  that $P_N\in A_N$ holds with probability $\Pro(A) \geq 1-c(n) N^{-1-\frac{4}{n-1}}$.
Notice that by Theorem \ref{affentranger:thm} we have
\begin{equation*}
	\lim_{N\to \infty} N^{\frac{2}{n-1}}\, (r_j-1) = \frac{\alpha(n,j)}{2}.
\end{equation*}
By \eqref{eqn:increasing_est} in Appendix \ref{sec:asymptotic_estimates}, $\alpha(n,j)$ is strictly increasing in $j$, so if $N$ is large enough then
\begin{equation*}
	1\leq r_1 \leq \dots \leq r_n.
\end{equation*}
Furthermore, by  \eqref{eqn:alpha_rel_est} in Appendix \ref{sec:asymptotic_estimates} we derive
\begin{equation*}
	\limsup_{N\to \infty} N^{\frac{2}{n-1}}\, \left(1-\left(\frac{r_1}{r_j}\right)^j\right) 
	\leq \frac{j\alpha(n,1)}{2} \left(\frac{\alpha(n,j)}{\alpha(n,1)}-1\right)\\
	\leq C_1 \frac{j\ln n}{n}
\end{equation*}
for some absolute constant $C_1>0$.
Hence,
\begin{equation*}
	V_j(D_n\cap r_1 P_N) 
	= r_1^j V_j\left(\frac{1}{r_1}D_n\cap P_N\right) 
	\geq r_1^j V_j\left(\frac{1}{r_j} D_n\cap P_N\right) 
	= \left(\frac{r_1}{r_j}\right)^j V_j(D_n \cap r_j P_N),
\end{equation*}
which yields
\begin{equation*}
	\Delta_j(D_n, r_1 P_N) 
	\leq \Delta_j(D_n, r_jP_N) + 2 V_j(D_n) \left(1-\left(\frac{r_1}{r_j}\right)^j\right).
\end{equation*}
Thus, if $N$ is sufficiently large, then for all $j\in [n]$ we have
\begin{equation*}
	\Delta_j(D_n, r_1 P_N) 
	\leq \underbrace{(c_4+2C_1)}_{=:C_2} \frac{j \ln n}{n} V_j(D_n) N^{-\frac{2}{n-1}}
\end{equation*}
with high probability for some absolute constant $C_2>0$.
Hence, if $N$ is sufficiently large, then with high probability
\begin{equation*}
	\Delta_\Sigma(D_n, r_1P_N) \leq \sum_{j=1}^n C_2 \frac{\ln n}{n}  jV_j(D_n) N^{-\frac{2}{n-1}} = C_2 \frac{\ln n}{n} \widehat{W}(D_n) N^{-\frac{2}{n-1}}.
\end{equation*}
Since this holds true with positive probability, if $N$ is large enough there exists a realization $Q_N$ of $r_1P_N$ which verifies the upper bound in Theorem \ref{willsthm} v). \qed


\subsection{Proof of Theorem \ref{gilwillthm}: An upper bound for the Wills deviation for convex bodies}\label{proofGilWill}

Let $K$ be a convex body in $\R^n$ that admits a rolling ball from the inside. For each $j\in[n]$, choose $M:=\lfloor N/n \rfloor$ points $X_{1}^j,\ldots,X_{M}^j$ at random from $\partial K$ independently and according to the optimal density $\phi_j$  defined in \eqref{eqn:bestdensity}, and let $P^j_M:=\conv\{X_{1}^j,\ldots,X_{M}^j\}$. Using Theorem \ref{reitznerthm} as in \eqref{reitznerbest1}, we derive that for each $j\in[n]$,
\begin{equation}\label{eqn:bound_P_M}
	\lim_{M\to\infty} M^{\frac{2}{n-1}}\, \E\Delta_j(K,P^j_M)
	= \beta(n,j) \left(\,\,\int_{\partial K} H_{n-j}(K,x)^{\frac{n-1}{n+1}} H_{n-1}(K,x)^{\frac{1}{n+1}}\,d\mu_{\partial K}(x)\right)^{\frac{n+1}{n-1}}.
\end{equation}
Now consider the random polytope $P_N$ defined by
\begin{equation*}
	P_{N} := \conv\left(\bigcup_{j=1}^n P^j_M\right) = \conv \left(\bigcup_{j=1}^n \{X_{1}^j,\ldots,X_{M}^j\}\right).
\end{equation*}
Note that $P_{N}$ has at most $nM \leq N$ vertices. Thus, since $P^j_M\subset P_{N}$ for all $j\in[n]$, from \eqref{eqn:bound_P_M} we obtain
\begin{align*}
	\limsup_{N\to\infty} N^{\frac{2}{n-1}}\, \E\Delta_\Sigma(K,P_{N}) 
	&\leq \sum_{j=1}^n \limsup_{N\to\infty} M^{\frac{2}{n-1}}\, \E\Delta_j(K,P_{M}^j)\\
 	&\leq n^{\frac{2}{n-1}} \sum_{j=1}^n
		\beta(n,j)\left(\,\,\int_{\partial K}H_{n-j}(K,x)^{\frac{n-1}{n+1}} H_{n-1}(K,x)^{\frac{1}{n+1}}\,d\mu_{\partial K}(x)\right)^{\frac{n+1}{n-1}}.
\end{align*}
From this bound on the expectation, it follows that there exists a realization $Q_{N}\in\cP_N^i(K)$ of the random polytope $P_{N}$ that also satisfies the bound. This concludes the proof.\qed

\section{Dual volume approximation of convex bodies}\label{sec:dual_thm}

For $K\in \cK_0(\R^n)$, the \emph{radial function} $\rho_K: \S^{n-1} \to (0, \infty)$ is defined by $\rho_K(u) = \max \{ r>0 : ru \in K\}$.
Integrating with respect to polar coordinates,  from \eqref{eqn:radial_steiner} we derive that
\begin{equation*}
	\widetilde{V}_j(K) = V_j(D_n) \int_{\Sp} \rho_K(u)^j \, d\sigma(u).
\end{equation*}
Since $V_j(D_n) = \binom{n}{j} \frac{|D_n|}{|D_{n-j}|}$, we define the analytic extension
\begin{equation}\label{eqn:ball_vol_analytic}
	V_q(D_n) :=  \frac{\pi^{\frac{q}{2}} \Gamma(n+1)}{\Gamma(q+1) \Gamma(n-q+1)} \frac{\Gamma(\frac{n-q}{2} + 1)}{\Gamma(\frac{n}{2}+1)}, \quad \forall q\in[0,n],
\end{equation}
and set
\begin{equation*}
	\widetilde{V}_q(K) := \begin{cases}
	                      	V_{|q|}(D_n) \int\limits_{\S^{n-1}} \rho_K(u)^q \, d\sigma(u) & \text{if $q\in [-n,n]$,}\\
	                      	|D_n| \int\limits_{\S^{n-1}} \rho_K(u)^q \, d\sigma(u) & \text{else.}
	                      \end{cases}
\end{equation*}
The quantity $\widetilde{V}_q(K)$ is finite since the origin lies in the interior of $K$, and therefore $\rho_K(u)>0$ for all $u\in \Sp$. For $q\in \R\setminus\{0\}$ and $K,L\in\cK_0(\R^n)$, we define the {\it $q$th dual volume deviation} $\widetilde{\Delta}_q(K,L)$ by
\begin{equation*}
	\widetilde{\Delta}_q(K,L) := \widetilde{V}_q(K) + \widetilde{V}_q(K) - 2 \widetilde{V}_q(K\cap L).
\end{equation*}
For $q=0$, we set
\begin{equation*}
	\widehat{V}_0(K) := \lim_{q\to 0} \frac{1}{|q|} \widetilde{V}_q(K) = \int_{\S^{n-1}} \ln \rho_K(u) \, d\sigma(u)
\end{equation*}
and define
\begin{equation*}
	\widehat{\Delta}_0(K,L) := \widehat{V}_0(K) + \widehat{V}_0(K) - 2 \widehat{V}_0(K\cap L).
\end{equation*}

We will need the following  lemma, which is related to \cite[Thm.\ 4.1]{GVV:2003}.
\begin{lemma}\label{lem:dual_to_weighted}
	Let $K,L\in \cK_0(\R^n)$.
	\begin{enumerate}
		\item[i)] If $0<|q|\leq n$, then
			\begin{equation}
				\widetilde{\Delta}_q(K,L) 
				= V_{|q|}(D_n) \int_{\Sp} \left| \rho_K(u)^q-\rho_L(u)^q\right| \, d\sigma(u) 
				= \frac{|q|V_{|q|}(D_n)}{n|D_n|} \int_{K\triangle L} \|x\|^{q-n} \, dx.
			\end{equation}
			
		\item[ii)] If $|q|>n$, then
			\begin{equation}
				\widetilde{\Delta}_q(K,L) 
				= |D_n| \int_{\Sp} \left| \rho_K(u)^q-\rho_L(u)^q\right| \, d\sigma(u) 
				= \int_{K\triangle L} \|x\|^{q-n} \, dx.
			\end{equation}
			
		\item[iii)] If $q=0$, then
			\begin{equation}\label{eqn:dDelta0}
				\widehat{\Delta}_0(K,L) = \int_{\S^{n-1}} \left| \ln \frac{\rho_K(u)}{\rho_L(u)}\right| \, d\sigma(u) 
				= \frac{1}{n|D_n|} \int_{K\triangle L} \|x\|^{-n} \, dx.
			\end{equation}
	\end{enumerate}
\end{lemma}
\begin{proof}
	First, let $0<|q|\leq n$.
	Since the origin is an interior point of $K$ and $L$, for all $u\in \Sp$ we have
	\begin{equation*}
		\rho_{K\cap L}(u) = \min\{\rho_K(u),\rho_L(u)\}.
	\end{equation*}
	Also, by the elementary fact that for any $a,b\in\R$ 
	\begin{equation*}
		a+b -2\min\{a,b\} = |a-b| = \max\{a,b\} - \min\{a,b\},
	\end{equation*}
	we conclude
	\begin{align*}
		\widetilde{\Delta}_q(K,L) 
		&= V_{|q|} (D_n) \int_{\Sp} \left(\rho_K(u)^q + \rho_L(u)^q - 2\rho_{K\cap L}(u)^q\right) \, d\sigma(u)\\
		&= V_{|q|} (D_n)  \int_{\Sp} \left| \rho_K(u)^q-\rho_L(u)^q\right| \, d\sigma(u)\\
		&= |q|V_{|q|} (D_n)  \int_{\Sp} \int_{\min\{\rho_K(u),\rho_L(u)\}}^{\max\{\rho_K(u),\rho_L(u)\}} t^{q-1}\, dt\, d\sigma(u).
	\end{align*}
	Using polar coordinates and the fact that
	\begin{equation*}
		K\triangle L = \big\{tu : u\in \Sp, \min\{\rho_K(u),\rho_L(u)\}\leq t \leq \max\{\rho_K(u),\rho_L(u)\}\big\},
	\end{equation*}
	we derive
	\begin{equation*}
		\widetilde{\Delta}_q(K,L) = \frac{|q|V_{|q|} (D_n) }{n|D_n|} \int_{K\triangle L} \|x\|^{q-n}\, dx.
	\end{equation*}
	The cases  $|q|>n$ and $q=0$ follow from similar arguments.
\end{proof}

\begin{remark}
	Let $K,L\in \cK_0(\R^n)$. Then $\widetilde{\Delta}_n(K,L)=\Delta_n(K,L)$.
	Furthermore, since $\rho_K=1/h_{K^\circ}$ we have
	\begin{equation*}
		\widetilde{\Delta}_{-1}(K,L) 
		= V_1(D_n) \int_{\Sp} \left|h_{K^\circ}(u)-h_{L^\circ}(u)\right| \, d\sigma(u)
		= V_1(D_n) \delta_1(K^\circ,L^\circ) \leq \Delta_1(K^\circ,L^\circ),
	\end{equation*}
	with equality if and only if $K^\circ\cup L^\circ = (K\cap L)^\circ$ (see Appendix \ref{sec:width_rel}).
	A similar observation was also made in \cite[Rmk.\ 3]{GlasauerGruber:1997}.
	
	For $q=0$, we use \eqref{eqn:dDelta0} and derive that
	\begin{equation*}
		\int_{\S^{n-1}} \left|\ln \frac{h_K(u)}{h_L(u)}\right| \, d\sigma(u) =\frac{1}{n|D_n|} \int_{K^\circ \triangle L^\circ} \|x\|^{-n}\, dx.
	\end{equation*}
\end{remark}

Recall that for a continuous function $\psi:\R^n\to (0,\infty)$, the $\psi$-weighted volume difference is defined by
\begin{equation*}
	\Delta_n^\psi(K,L) := \int_{K\triangle L} \psi(x) \, dx.
\end{equation*}
By Lemma \ref{lem:dual_to_weighted}, we can identify approximation in the $q$th dual volume deviation with approximation under the $\psi$-weighted volume difference.
The following results on weighted volume approximation were obtained by Ludwig \cite{Ludwig:1999} for convex bodies of class $C_+^2$ and extended to $C^2$ by B\"or\"oczky \cite{Boroczky:2000a}. Note that Ludwig's proof for $\cC_N=\cP_N$ and $\cC_N=\cP_{(N)}$ also gives the corresponding results for the inscribed case $\cC_N= \cP_N^i(K)$ (see \cite{Gruber:1993a}) and the circumscribed case $\cC_N = \cP_N^o(K)$ (see \cite{GlasauerGruber:1997}), where only the dimensional constants in the limit need to be changed.

\begin{theorem}[{\cite{Ludwig:1999}}]\label{thm:weighted_approx}
	Let $K\subset \R^n$ be a convex body of class $C^2$. 
	Furthermore, let $\cC_N$ be either $\cP_N^i(K)$, $\cP_N$, $\cP_{(N)}^o(K)$, or $\cP_{(N)}$, and set $\gamma_{n-1}$ as in \eqref{eqn:c_del_div}.
	If $\psi:\R^n\to (0,\infty)$ is a  continuous function, then
	\begin{equation}\label{eqn:weighted_approx}
		\lim_{N\to \infty} N^{\frac{2}{n-1}}  \, \Delta_n^{\psi}(K, \cC_N) 
		= \frac{\gamma_{n-1}}{2}\, \left(\,\int_{\partial K} \psi(x)^{\frac{n-1}{n+1}} H_{n-1}(K,x)^{\frac{1}{n+1}} \, d\mu_{\partial K}(x) \right)^{\frac{n+1}{n-1}}.
	\end{equation}
\end{theorem}

We are now ready to prove the following extension of Theorem \ref{dualthm}.
\begin{theorem}\label{thm:dualthm_gen}
	Let $K\subset \R^n$ be a convex body of class $C^2$ that contains the origin in its interior, let $\cC_N$ be either $\cP_N^i(K)$, $\cP_N$, $\cP_{(N)}^o(K)$, or $\cP_{(N)}$, and set $\gamma_{n-1}$ as in \eqref{eqn:c_del_div}.
	For $q\in \R$, set
	\begin{equation*}
		\Omega_q(K) := \int_{\partial K} \|x\|^{\frac{(q-n)(n-1)}{n+1}} H_{n-1}(K,x)^{\frac{1}{n+1}} \, d\mu_{\partial K}(x).
	\end{equation*}
	\begin{enumerate}
		\item[i)] If $0<|q|\leq n$, then 
			\begin{equation*}
				\lim_{N\to \infty} N^{\frac{2}{n-1}}\,\widetilde{\Delta}_q(K,\cC_N) 
					= \frac{\gamma_{n-1}}{2} \frac{|q| V_{|q|}(D_n)}{n|D_n|}\, \Omega_q(K)^{\frac{n+1}{n-1}}.
			\end{equation*}
	
		\item[ii)] If $|q|>n$, then
			\begin{equation*}
				\lim_{N\to \infty} N^{\frac{2}{n-1}}\,\widetilde{\Delta}_q(K,\cC_N) 
					= \frac{\gamma_{n-1}}{2}\, \Omega_q(K)^{\frac{n+1}{n-1}}.
			\end{equation*}
			
		\item[iii)] If $q=0$, then
			\begin{equation*}
				\lim_{N\to \infty} N^{\frac{2}{n-1}}\,\widehat{\Delta}_0(K,\cC_N) 
					= \frac{\gamma_{n-1}}{2} \frac{1}{n|D_n|}\, \Omega_0(K)^{\frac{n+1}{n-1}}.
			\end{equation*}
	\end{enumerate}
\end{theorem}
\begin{proof}
Let $0<|q|<n$.
We would like to use the weight function $x\mapsto \|x\|^{q-n}$, which is discontinuous at the origin if $q<n$. We remedy this by the following standard arguments. Let $K\in\mathcal K_0(\R^n)$. Then there exists $\varepsilon_0>0$ such that for all $0<\varepsilon\leq \varepsilon_0$, the ball $\varepsilon D_n$ is also contained in the interior of $K$. We define the continuous weight function $\psi_\varepsilon$ by
\begin{equation*}
	\psi_{\varepsilon}(x) := \begin{cases}
	                        	\|x\|^{q-n} & \text{if $\|x\|>\varepsilon$},\\
	                        	\varepsilon^{q-n} & \text{if $\|x\|\leq \varepsilon$}.
	                        \end{cases}
\end{equation*}
Notice that $\psi_{\varepsilon}(x) = \|x\|^{q-n}$ for all $x\in \partial K$ and $\varepsilon \in (0,\varepsilon_0)$. Thus,
\begin{equation*}
	\Omega_q(K) = \int_{\partial K} \psi_\varepsilon(x)^{\frac{n-1}{n+1}} H_{n-1}(K,x)^{\frac{1}{n+1}}\, d\mu_{\partial K}(x), \quad \forall \varepsilon\in (0,\varepsilon_0),
\end{equation*}
where we extended the definition \eqref{eqn:dual_curvature_measure} of $\Omega_j$ to $q\in\R$.
By Lemma \ref{lem:dual_to_weighted}, if $P$ is a polytope that contains $\varepsilon D_n$ then
\begin{equation*}
	\widetilde{\Delta}_q(K,P) = \frac{|q| V_{|q|}(D_n)}{n|D_n|} \Delta_n^{\psi_{\varepsilon}}(K,P).
\end{equation*}
Observe that if $N$ is large enough, then any polytope $P\in \cC_N$ that is well-approximating contains $\varepsilon D_n$. To see this,
 for $u\in\S^{n-1}$ define the closed halfspace $H_\varepsilon^+(u) := \{x\in \R^n:  x \cdot u  \geq \varepsilon\}$ and set
\begin{equation*}
	\eta=\eta(K,q) := \min_{u\in \S^{n-1}} \widetilde{V}_q(K\cap H_\varepsilon^+(u)) > 0.
\end{equation*}
If $P$ does not contain $\varepsilon D_n$, then there exists $u\in \S^{n-1}$ such that $P\cap H_\varepsilon^+(u) =\varnothing$. In this case,
\begin{equation*}
	\widetilde{\Delta}_q(K,P) \geq \widetilde{V}_q(K) - \widetilde{V}_q(K\cap P) \geq \widetilde{V}_q(K\cap H_\varepsilon^+(u)) \geq \eta >0.
\end{equation*}
Since
\begin{equation*}
	\widetilde{V}_q(K\cap H_\varepsilon^+(u)) = \frac{|q| V_{|q|}(D_n)}{n|D_n|} \int_{K\cap H_\varepsilon^+(u)} \|x\|^{q-n}\, dx,
\end{equation*}
we also find that $\frac{|q| V_{|q|}(D_n)}{n|D_n|} \Delta_n^{\psi_{\varepsilon}}(K,P) \geq \eta$.
Furthermore, since $\frac{|q| V_{|q|}(D_n)}{n|D_n|} \Delta_n^{\psi_{\varepsilon}}(K,\cC_N)\to 0$ and $\widetilde{\Delta}_q(K,\cC_N) \to 0$ as $N\to \infty$, there exists $N_0>0$ such that for all $N\geq N_0$ we have $\widetilde{\Delta}_q(K,\cC_N) < \eta$ and $\frac{|q| V_{|q|}(D_n)}{n|D_n|} \Delta_n^{\psi_{\varepsilon}}(K,\cC_N)< \eta$. Thus, if $N\geq N_0$ then
\begin{align*}
	\widetilde{\Delta}_q(K,\cC_N) 
	&= \min \left\{ \widetilde{\Delta}_q(K,P): P\in \cC_N \text{ and $\varepsilon D_n\subset P$}\right\} \\
	&= \frac{|q| V_{|q|}(D_n)}{n|D_n|} \min \left\{ \Delta_n^{\psi_{\varepsilon}}(K,P) : P\in \cC_N \text{ and $\varepsilon D_n\subset P$}\right\} 
	= \frac{|q| V_{|q|}(D_n)}{n|D_n|} \Delta_n^{\psi_{\varepsilon}}(K,\cC_N).
\end{align*}
Therefore, by Theorem \ref{thm:weighted_approx} we have 
\begin{equation*}
	\lim_{N\to \infty} N^{\frac{2}{n-1}}\, \widetilde{\Delta}_q(K,\cC_N)   
	=\frac{|q|V_{|q|}(D_n)}{n|D_n|}  \lim_{N\to\infty}  N^{\frac{2}{n-1}}\, \Delta_n^{\psi_\varepsilon}(K,\cC_N)  
	= \frac{\gamma_{n-1}}{2} \frac{|q|V_{|q|}(D_n)}{n|D_n|} \, \Omega_q(K)^{\frac{n+1}{n-1}},
\end{equation*}
which concludes the proof for all $0<|q|\leq n$. The case $|q|>n$ follows analogously, and the case $q=0$ follows from  \eqref{eqn:dDelta0}.
\end{proof}

\subsection{The dual Wills functional and dual Wills deviation}\label{sec:dualWills}

For $K\in \cK(\R^n)$, the Wills functional $W(K)=\sum_{j=0}^n V_j(K)$  can be expressed as
\begin{equation}\label{willsexp}
	W(K) = \int_{\R^n} e^{-\pi \dist(x,K)^2} dx ,
\end{equation}
where $\dist(x,K) = \min_{y\in K} \|x-y\|$ (see, e.g., \cite{Vitale:1996}).
For $K\in \cK_0(\R^n)$ and $x\in \R^n$, the \emph{minimal radial distance} $\rdist(x,K)$ of $x$ to $K$ is defined
as the minimal distance between $x$ and a point $y\in K$ such that $x,y$ and the origin are collinear, or equivalently,
\begin{equation*}
	\rdist(x,K) := \begin{cases}
	              	\|x\|-\rho_K(x/\|x\|) & \text{if $x\not \in K$,}\\
	              	0 & \text{if $x\in K$.}
	              \end{cases}
\end{equation*}
Note that 
\begin{equation*}
	K\,\tilde{+}\, r D_n = \{x \in \R^n : \rdist(x,K)\leq r \}.
\end{equation*}
Hence, by the radial Steiner formula \eqref{eqn:radial_steiner} it follows that
\begin{equation*}
	\int_{\R^n} e^{-\pi \rdist(x,K)^2} dx = \int_{0}^\infty |K\,\tilde{+}\,t D_n| e^{-\pi t^2} dt = \sum_{j=0}^n \dV_j(K).
\end{equation*}
This motivates us to define the \emph{dual Wills functional} $\dW$ by
\begin{equation}
	\dW(K):= \int_{\R^n} e^{-\pi \rdist(x,K)^2} dx = \sum_{j=0}^n \dV_j(K), \quad  \text{for all $K\in\cK_0(\R^n)$.}
\end{equation}
For $K,L\in \cK_0(\R^n)$, we also define the \emph{dual Wills deviation} $\widetilde\Delta_\Sigma(K,L)$ by
\begin{equation}
	\widetilde\Delta_\Sigma(K,L) := \dW(K)+ \dW(L)-2 \dW(K\cap L) = \sum_{j=1}^n\widetilde{\Delta}_j(K,L).
\end{equation}

\paragraph{Proof of Theorem \ref{dualwillscor}.}
	Since
	\begin{equation*}
		\widetilde{\Delta}_{\Sigma}(K,\cC_N) = \min_{P\in \cC_N} \sum_{j=1}^n \widetilde{\Delta}_j(K,P) \geq \sum_{j=1}^n \widetilde{\Delta}_{j}(K,\cC_N),
	\end{equation*}
	the lower bound \eqref{eqn:dualWills_lower} follows from Theorem \ref{dualthm}.
	
	\medskip
	For the upper bound, we follow arguments similar to those in Subsection \ref{proofGilWill} using \eqref{eqn:random_weighted_approx_best}. For $j\in[n]$, we therefore define the continuous, positive and bounded weight function 
	\begin{equation*}
		\psi_{j}(x)
		:= \begin{cases} 
			\|x\|^{j-n} & \text{ if $\|x\|>\varepsilon$},\\
			\varepsilon^{j-n} &\text{ if $\|x\|\leq \varepsilon$},
		\end{cases}
	\end{equation*}
	where we assume that $\varepsilon>0$ is so small that $\varepsilon D_n$ is contained in the interior of $K$.
	Now by Lemma \ref{lem:dual_to_weighted} we have
	\begin{equation*}
		\widetilde{\Delta}_j(K,L) = \frac{jV_j(D_n)}{n|D_n|} \Delta_n^{\psi_j}(K,L)
	\end{equation*}
	for all $K,L\in \cK_0(\R^n)$ such that $\varepsilon D_n \subset K\cap L$.
	Then the density function corresponding to $\psi$ on $\partial K$ that minimizes the right-hand side of \eqref{eqn:random_weighted_approx} and \eqref{eqn:random_approx_circum_weighted} is given by
	\begin{equation*}
		\widetilde{\psi}_j(x) = \frac{1}{\Omega_j(K)} \|x\|^{\frac{(j-n)(n-1)}{n+1}} H_{n-1}(K,x)^{\frac{1}{n+1}}, \quad \forall x\in \partial K,
	\end{equation*}
	see \eqref{eqn:optimal_psi}.
	Thus,
	\begin{equation*}
		\lim_{N\to \infty} N^{\frac{2}{n-1}}\,\E \Delta_j(K,P_N^{\widetilde{\psi}_j}) 
		= \frac{jV_j(D_n)}{n|D_n|} \lim_{N\to \infty}N^{\frac{2}{n-1}}\, \E\Delta_n^{\psi_j}(K,P_N^{\widetilde{\psi}_j})
		= \frac{\beta(n,n)}{n|D_n|} jV_j(D_n) \Omega_j(K)^{\frac{n+1}{n-1}}
	\end{equation*}
	and 
	\begin{equation*}
		\lim_{N\to \infty} N^{\frac{2}{n-1}}\,\E \Delta_j(K,P_{(N)}^{\widetilde{\psi}_j}) 
		= \frac{jV_j(D_n)}{n|D_n|} \lim_{N\to \infty}N^{\frac{2}{n-1}}\, \E\Delta_n^{\psi_j}(K,P_{(N)}^{\widetilde{\psi}_j})
		= \frac{\beta(n,1)}{V_1(D_n)} jV_j(D_n) \Omega_j(K)^{\frac{n+1}{n-1}}.
	\end{equation*}
	Now we use a  random construction similar to the one in Subsection \ref{proofGilWill}. For each $j\in[n]$, choose $X_{1}^j,\ldots, X^j_{\lfloor N/n\rfloor}$ from $\partial K$ independently and according to the density $\widetilde{\psi}_j$, and set 
	\begin{equation*}
		P_{N} := \conv \left(\bigcup_{j=1}^n \left\{X_{1}^j,\dotsc, X^j_{\lfloor N/n\rfloor}\right\}\right). 
	\end{equation*}
	Then $P_{\lfloor N/n\rfloor}^j := \conv\left\{X^j_1,\dotsc,X^j_{\lfloor N/n\rfloor}\right\}\subset P_N$ for all $j\in [n]$, and therefore
	\begin{align*}
		\limsup_{N\to \infty} N^{\frac{2}{n-1}}\, \widetilde{\Delta}_{\Sigma}(K,P_{N})
		&\leq \sum_{j=1}^n \frac{jV_j(D_n)}{n|D_n|} \limsup_{N\to\infty} N^{\frac{2}{n-1}} \, \E \Delta_n^{\psi_j}\left(K,P_{\lfloor N/n \rfloor}^j \right) \\
		&\leq n^{\frac{2}{n-1}} \frac{\beta(n,n)}{n|D_n|} \sum_{j=1}^n jV_j(D_n) \Omega_j(K)^{\frac{n+1}{n-1}}.
	\end{align*}
	Similarly, we define $P_{(N)}$ by
	\begin{equation*}
		P_{(N)} := \bigcap_{j=1}^n \left( H_{j,1}^- \cap \dotsc \cap H_{j,\lfloor N/n\rfloor}^-\right),
	\end{equation*}
	where $X_{i}^j$ is distributed with respect to $\widetilde{\psi}_j$ for $j\in[n]$ and $1\leq i \leq \lfloor N/n\rfloor$, and
	$H_{j,i}^-$ is the supporting halfspace of $K$ at $X_{i}^j$. Notice that $P_{(\lfloor N/n\rfloor)}^j := H_{j,1}^-\cap \dotsc \cap H_{j,\lfloor N/n \rfloor}^- \supset P_{(N)}$, 
	and by Theorem \ref{thm:circumscribed} we conclude
	\begin{align*}
		\limsup_{N\to \infty} N^{\frac{2}{n-1}}\, \E\widetilde{\Delta}_{\Sigma}(K,P_{(N)} &\cap (K+D_n))
		\\
&\leq \sum_{j=1}^n \frac{jV_j(D_n)}{n|D_n|} \limsup_{N\to\infty} N^{\frac{2}{n-1}} \, \E \Delta_n^{\psi_j}\left(K,P_{(\lfloor N/n\rfloor)}^j\cap (K+D_n)\right) \\
		&\leq n^{\frac{2}{n-1}} \frac{\beta(n,1)}{V_1(D_n)} \sum_{j=1}^n jV_j(D_n) \Omega_j(K)^{\frac{n+1}{n-1}}.
	\end{align*}
Hence, there is a realization $Q_{N} \in \cP_N^i(K)$ of $P_{N}$ such that
	\begin{equation*}
		\limsup_{N\to \infty} N^{\frac{2}{n-1}}\, \widetilde{\Delta}_{\Sigma}(K,\cP_N^i(K))
		\leq \limsup_{N\to\infty} N^{\frac{2}{n-1}}\, \widetilde{\Delta}_{\Sigma}(K,Q_{N})
		\leq n^{\frac{2}{n-1}} \frac{\beta(n,n)}{n|D_n|} \sum_{j=1}^n jV_j(D_n) \Omega_j(K)^{\frac{n+1}{n-1}},
	\end{equation*}
	and similarly there is a realization $Q_{(N)}\in\cP_{(N)}^o(K)$ of $P_{(N)}$ such that $Q_{(N)}\subset K+D_n$ and
	\begin{equation*}
		\limsup_{N\to \infty} N^{\frac{2}{n-1}} \widetilde{\Delta}_{\Sigma}(K,\cP_{(N)}^o(K))
		\leq \limsup_{N\to\infty} N^{\frac{2}{n-1}} \widetilde{\Delta}_{\Sigma}(K,Q_{(N)})
		\leq n^{\frac{2}{n-1}} \frac{\beta(n,1)}{V_1(D_n)} \sum_{j=1}^n jV_j(D_n) \Omega_j(K)^{\frac{n+1}{n-1}}.
	\end{equation*}
	Therefore, the upper bound \eqref{eqn:dualWills_upper} holds true. \qed


\section{The (dual) stochastic Wills deviation}\label{sec:stochasticwills}

Vitale \cite{Vitale:1996} generalized the Wills functional (and, in a sense, the Steiner formula) using a probabilistic construction. This construction was also considered by Hadwiger \cite{Hadwiger:1975} without the probabilistic notation (see also \cite{Kampf:2009, McMullen:1991}). Consider a random variable $\Lambda$ on $[0,\infty)$ with $\E \Lambda^n <\infty$. Using $\Lambda$ as the radius $r$ in the Steiner formula \eqref{eqn:steiner} and taking the expectation, we obtain
\begin{equation}\label{stochasticwills}
	W_\Lambda(K):=\E |K + \Lambda D_n| =
	\int_{\R^n} \Pro(\dist(x,K)\leq \Lambda)\, dx 
	=\sum_{j=0}^n V_j(K)|D_{n-j}| \, \E \Lambda^{n-j}.
\end{equation}
For a nonnegative random variable $\Lambda$ with finite $n$th moment, we call $W_{\Lambda}$ the \emph{$\Lambda$-Wills functional}, or stochastic Wills functional, which is defined for any convex body $K\in\cK(\R^n)$. In particular, if $\Lambda$ is the constant $r>0$ then we  recover the Steiner formula \eqref{eqn:steiner}.
Furthermore, if $\Lambda=\Sigma$ where $\Sigma$ is a  Weibull random variable with parameters $(k,\lambda) = (2,\sqrt{1/\pi})$, i.e., $\Sigma$ has the density $2\pi t e^{-\pi t^2}$ for $t\geq 0$, then $\E \Sigma^k = |D_k|^{-1}$ and we recover the Wills functional of $K$:
\begin{equation*}
	W_\Sigma(K) = \E |K + \Sigma D_n| = \int_{\R^n} e^{-\pi \dist(x,K)^2} dx = \sum_{j=0}^n V_j(K) = W(K).
\end{equation*}
 Now, given $K,L\in\cK(\R^n)$ and a nonnegative random variable $\Lambda$ with $\E \Lambda^n <\infty$, we  define the \emph{$\Lambda$-Wills deviation} $\Delta_\Lambda(K,L)$ by
\begin{equation}
	\Delta_{\Lambda}(K,L):= W_\Lambda(K)+W_\Lambda(L)-2W_\Lambda(K\cap L) = \sum_{j=1}^n \Delta_{j}(K,L)|D_{n-j}|\,\E \Lambda^{n-j}.
\end{equation}
Note that the stochastic Wills functional is not a metric in general; a proof is given in Appendix \ref{sec:not_a_metric}.

\smallskip

We  derive an immediate corollary to Theorem \ref{willsthm}.
\begin{corollary}\label{genwillsthm}
	Set $\widehat W_\Lambda(D_n):=\sum_{j=1}^n j V_j(D_n)|D_{n-j}|\,\E \Lambda^{n-j}$. 
	Then with the same absolute constants $c_1,c_2,c_3,c_7>0$ from Theorem \ref{willsthm} and for all sufficiently large $N$, the following estimates hold true:
	\begin{align*}
		\begin{array}{lrclclr}
		i)  & c_1 \widehat W_\Lambda(D_n)  &\!\!\!\!\leq\!\!\!\!& N^{\frac{2}{n-1}}\, \Delta_\Lambda(D_n,\cC_N)       &\!\!\!\!\leq\!\!\!\!& c_2 \widehat W_\Lambda(D_n), 
			& \text{for $\cC_N\in\{\cP_N^{i},\cP_{(N)}^{o}\}$;}\\
		ii) & c_1 \widehat W_\Lambda(D_n)  &\!\!\!\!\leq\!\!\!\!& N^{\frac{2}{n-1}}\, \Delta_\Lambda(D_n,\cP_{k,N}^{i,s}), &&
			& \text{for $k\in\{0,1,\ldots,\lfloor n/2\rfloor\}$;}\\
		iii)& c_1 \widehat W_\Lambda(D_n)  &\!\!\!\!\leq\!\!\!\!& N^{\frac{2}{n-1}}\, \Delta_\Lambda(D_n,\cP_{k,N}^{o,s}), &&
			& \text{for $k\in\{\lfloor 3(n-1)/4\rfloor,\ldots,n-1\}$;}\\
		iv)&&                                   & N^{\frac{2}{n-1}}\, \Delta_\Lambda(D_n,\cP_{(N)}^i) &\!\!\!\!\leq\!\!\!\!& c_3 \widehat W_\Lambda(D_n);\\
		v) &&                                   & N^{\frac{2}{n-1}}\, \Delta_\Lambda(D_n,\cP_N)       &\!\!\!\!\leq\!\!\!\!& c_7\frac{\ln n}{n} \widehat W_\Lambda(D_n).
		\end{array}
	\end{align*}
	Furthermore, the bound in i) for $k=1$ and the bound in ii) for $k=n-2$ also hold true for nonsimplicial polytopes.
\end{corollary}
\begin{proof}
In the proof of Theorem \ref{willsthm} i), it was shown that for $\cC_N\in\{\cP_N^i, \cP_{(N)}^o\}$ there exists a polytope $Q\in\cC_N$ which satisfies all of the upper bounds in Theorem \ref{thm:1} i) simultaneously. Thus, for all sufficiently large $N$,
\begin{equation*}
	\Delta_\Lambda(D_n,\cC_N) 
	\leq \Delta_\Lambda(D_n,Q) 
	= \sum_{j=1}^n\Delta_{j}(D_n,Q)|D_{n-j}|\E \Lambda^{n-j}
	\leq c_2 \widehat W_\Lambda(D_n) N^{-\frac{2}{n-1}},
\end{equation*}
so the upper bound in i) holds. Part iv) follows in the same way, but we use Theorem \ref{willsthm} iv) instead. The lower bounds in i), ii) and iii) follow from the lower bounds in Theorem \ref{thm:1} i), ii) and iii), respectively, together with the estimate
\begin{equation*}
	\Delta_\Lambda(D_n,\cC_N) 
	\geq \sum_{j=1}^n \Delta_{j}(D_n,\cC_N)|D_{n-j}| \E\Lambda^{n-j} 
	\geq c_1\,\widehat W_\Lambda(D_n) N^{-\frac{2}{n-1}}.
\end{equation*}
Here the last inequality holds when $N$ is large enough, and $c_1>0$ is the absolute constant from Theorem \ref{thm:1}.  
Finally, v) follows along the same lines as the proof of Theorem \ref{willsthm} v); we leave the details to the interested reader.
\end{proof}

\begin{remark}
Similar to Remark \ref{remarkWills}, we notice that
\begin{equation}
	\widehat{W}_{\Lambda}(D_n) = n|D_n| \sum_{k=0}^{n-1} \binom{n-1}{k} \E\Lambda^k = V_1(D_n) W_{\Lambda}(D_{n-1}).
\end{equation}
\end{remark}

\begin{remark}
Theorem \ref{gilwillthm} can be generalized to give an upper bound for the $\Lambda$-Wills deviation of a convex body with a rolling ball and an inscribed polytope with a prescribed number of vertices. We leave the details to the interested reader.
\end{remark}

\subsection{The dual stochastic Wills deviation}

For $K\in\mathcal K_0(\R^n)$, we define the \emph{dual $\Lambda$-Wills functional} $\dW_\Lambda(K)$ by 
\begin{equation*}
	\dW_\Lambda(K) := \E |K\, \tilde{+}\, \Lambda D_n| = \int_{\R^n} \Pro(\rdist(x,L)\leq \Lambda)\, dx = \sum_{j=0}^n \dV_{j}(K) |D_{n-j}|\, \E \Lambda^{n-j},
\end{equation*}
and if $L\in\mathcal K_0(\R^n)$ as well, we define the \emph{dual $\Lambda$-Wills deviation} $\widetilde\Delta_{\Lambda}(K,L)$ by
\begin{equation*}
	\widetilde\Delta_{\Lambda}(K,L):=\widetilde W_\Lambda(K)+\widetilde W_\Lambda(L)-2\widetilde W_\Lambda(K\cap L) 
	= \sum_{j=1}^n \widetilde\Delta_{j}(K,L)|D_{n-j}|\,\E \Lambda^{n-j}.
\end{equation*}
As a corollary to Theorem \ref{dualwillscor}, we obtain the following result.
\begin{corollary}\label{dualwillscor_stochastic}
	Let $K$ be a convex body of class $C^2$ that contains the origin in its interior. Furthermore, let $\cC_N$ be either $\cP_N^i$, $\cP_N^i(K)$, $\cP_{(N)}^o$, or $\cP_{(N)}$, and set $\gamma_{n-1}$ as in \eqref{eqn:c_del_div}. Then
	\begin{equation}\label{eqn:dualWills_lower_stochastic}
		\liminf_{N\to \infty} N^{\frac{2}{n-1}} \, \widetilde{\Delta}_{\Lambda}(K,\cC_N) 
			\geq \frac{\gamma_{n-1}}{2} \sum_{j=1}^n \binom{n-1}{j-1} \E\Lambda^{n-j} \, \Omega_j(K)^{\frac{n+1}{n-1}}.
	\end{equation}
	Moreover, if $K$ is a convex body that admits a rolling ball from the inside and contains the origin in its interior, then
	\begin{equation} \label{eqn:dualWills_upper_stochastic}
		\limsup_{N\to \infty} N^{\frac{2}{n-1}} \, \widetilde{\Delta}_{\Lambda}(K,\cP_N^i(K)) 
			\leq n^{\frac{2}{n-1}}\beta(n,n) \sum_{j=1}^n \binom{n-1}{j-1} \E\Lambda^{n-j} \, \Omega_j(K)^{\frac{n+1}{n-1}},
	\end{equation}
	and if $K$ is a convex body of class $C^2$, then
	\begin{equation}
		\limsup_{N\to \infty} N^{\frac{2}{n-1}} \, \widetilde{\Delta}_{\Lambda}(K,\cP_{(N)}^o(K)) 
			\leq n^{\frac{2}{n-1}}\beta(n,1) |D_{n-1}| \sum_{i=1}^n \binom{n-1}{j-1} \E\Lambda^{n-j} \, \Omega_j(K)^{\frac{n+1}{n-1}}.
	\end{equation}
	Furthermore, these upper bounds also hold true for $\cP_N$ and $\cP_{(N)}$ since $\cP_N \subset \cP_N^i(K)$ and $\cP_{(N)} \subset \cP_{(N)}^o(K)$.
\end{corollary}
\begin{proof}
	The corollary follows  analogously to Theorem \ref{dualwillscor}, where in the end we just notice that
	\begin{equation*}
		\frac{1}{n|D_n|} \sum_{j=1}^n \E\Lambda^{n-j} |D_{n-j}| j V_j(D_n) \, \Omega_j(K)^{\frac{n+1}{n-1}} 
		= \sum_{j=1}^n \binom{n-1}{j-1} \E\Lambda^{n-j} \, \Omega_j(K)^{\frac{n+1}{n-1}}.\qedhere
	\end{equation*}
\end{proof}

          
\section*{Acknowledgments}
The authors would like to express their sincerest gratitude to Ferenc Fodor, Daniel Hug,
Monika Ludwig, Matthias Reitzner, Christoph Th\"ale, Viktor V\'igh and Elisabeth Werner for the
enlightening discussions. The authors would also like to thank the anonymous referees for their
valuable comments that helped improve the article. The work of the third author is supported by the
Center for Brains, Minds and Machines (CBMM), funded by NSF STC award CCF-1231216.

\appendix

\section{\texorpdfstring{The intrinsic volume deviation and the stochastic Wills functional are not metrics}{The intrinsic volume deviation and the (stochastic) Wills functional are not metrics}}\label{sec:not_a_metric}

It is well-known that the symmetric volume difference $\Delta_n$ is a metric on $\cK(\R^n)$ and induces the same topology as the Hausdorff metric on the space of all convex bodies that have nonempty interiors (see, e.g., \cite{ShephardWebster:1965}). On the other hand, the surface area deviation $\Delta_{n-1}$ is not a metric since it does not satisfy the triangle inequality. An explicit counterexample was given in \cite{Hoehner:2016}. The same triple of convex bodies can be used to show that for any $j\in[n-1]$, the $j$th intrinsic volume deviation fails to satisfy the triangle inequality. 

\begin{lemma}\label{Deltajnotmetric}
	Fix $n\in\N$ with $n\geq 2$. For $j\in[n-1]$, the intrinsic volume deviation $\Delta_j$ does not satisfy the triangle inequality and is thus not a metric on $\cK(\R^n)$. 
	Moreover,  the stochastic Wills deviation $\Delta_\Lambda$ does not satisfy the triangle inequality and is thus not a metric on $\cK(\R^n)$. 
\end{lemma}
\begin{proof}
	Let $\varepsilon \in (0,1)$. We consider the disjoint caps $L_{\pm \varepsilon}$ of the ball $D_n$ defined by
	\begin{equation*}
		L_{\varepsilon} := \{(x_1,\dotsc,x_n) \in D_n : x_n \geq \varepsilon\}, \quad L_{-\varepsilon} := -L_{\varepsilon}.
	\end{equation*}
	Since $L_{\pm \varepsilon} \subset D_n$ and $V_j(L_{\varepsilon})=V_j(L_{-\varepsilon})$, we have
	\begin{equation*}
		\Delta_j(D_n,L_{\varepsilon}) = V_j(D_n) - V_j(L_{\varepsilon}) = \Delta_j(D_n,L_{-\varepsilon}).
	\end{equation*}
	Since $L_{\varepsilon} \cap L_{-\varepsilon} = \varnothing$, we also have
		$\Delta_j(L_\varepsilon,L_{-\varepsilon}) = 2V_j(L_\varepsilon)$.
	We want to show that
		$\Delta_j(D_n,L_{\varepsilon}) + \Delta_j(D_n,L_{-\varepsilon}) < \Delta_j(L_{\varepsilon},L_{-\varepsilon})$,
	which is equivalent to
	\begin{equation}\label{notametric}
		V_j(D_n) < 2 V_j(L_{\varepsilon}) = V_j(L_{\varepsilon}) + V_j(L_{-\varepsilon}).
	\end{equation}
	Since the caps $L_{\pm \varepsilon}$ converge to  half-balls as $\varepsilon\to 0^+$, by the continuity of the intrinsic volume $V_j$ and its valuation property we obtain
	\begin{equation*}
		\lim_{\varepsilon\to 0^+} \big(V_j(L_{\varepsilon}) + V_j(L_{-\varepsilon})\big) = V_j(D_n) + V_j(D_{n-1}) > V_j(D_n), \quad \forall j\in[n-1].
	\end{equation*}
	Thus, since $V_j(L_{\varepsilon})$ increases monotonically as $\varepsilon\to 0^+$, there exists $\varepsilon_0>0$ such that  \eqref{notametric} holds for all $\varepsilon\in (0,\varepsilon_0)$. 
	We have thus verified that $\Delta_j$ is not a metric.
	
	\smallskip
	Similarly, for the stochastic Wills deviation we want to show that
	$\Delta_\Lambda(D_n,L_\varepsilon) + \Delta_\Lambda(D_n,L_{-\varepsilon}) < \Delta_\Lambda(L_\varepsilon,L_{-\varepsilon})$,
	which holds if and only if
	\begin{equation}\label{notametricwills}
		W_{\Lambda}(D_n) < 2 W_{\Lambda}(L_\varepsilon).
	\end{equation}
	By \eqref{notametric},  for all $\varepsilon \in (0,\varepsilon_0)$ it holds that
	\begin{equation*}
		\sum_{j=1}^{n-1}V_j(L_\varepsilon)|D_{n-j}| \E \Lambda^{n-j} > \frac{1}{2}\sum_{j=1}^{n-1} V_j(D_n) |D_{n-j}| \E \Lambda^{n-j}.
	\end{equation*}
	Therefore,
	\begin{align}\label{willsnotmetric}
	\begin{split}
		W_{\Lambda}(L_\varepsilon) 
		&= |D_n|\E\Lambda^n + \sum_{j=1}^{n-1}V_j(L_\varepsilon)|D_{n-j}| \E\Lambda^{n-j} + |L_\varepsilon|\\
		&> |D_n|\E\Lambda^n + \frac{1}{2} \sum_{j=1}^{n-1} V_j(D_n)|D_{n-j}| \E\Lambda^{n-j} + |L_\varepsilon|\\
		&=\frac{1}{2}W_{\Lambda}(D_n)+ \frac{1}{2}|D_n|\E\Lambda^n + |L_\varepsilon|-\frac{1}{2}|D_n|.
	\end{split}
	\end{align}
	Since $\Lambda$ is a positive random variable, we have $\E \Lambda^n  >0$.
	Furthermore, since $|L_{\varepsilon}|$ increases to $|L_{0}| = \frac{1}{2}|D_n|$ as $\varepsilon \to 0^+$, there exists $\varepsilon_1 \in (0,\varepsilon_0)$ such that 
	\begin{equation*}
		|L_{\varepsilon_1}| > \frac{1}{2}|D_n| (1-\E\Lambda^n),
	\end{equation*}
	or equivalently, $\frac{1}{2}|D_n|\E\Lambda^n + |L_{\varepsilon_1}|-\frac{1}{2}|D_n|> 0$. Thus, from  \eqref{willsnotmetric} we obtain 
		$W_\Lambda(L_{\varepsilon_1}) > \frac{1}{2}W_\Lambda(D_n)$,
	which  proves \eqref{notametricwills}.
\end{proof}
In this proof we essentially used the discontinuity of the intrinsic volume deviation and of the stochastic Wills functional on $\cK(\R^n)$.
Notice that if $V$ is a continuous valuation on $\cK(\R^n)$, then the deviation functional $\Delta(K,L):=V(K)+V(L)-2V(K\cap L)$ is continuous on convex bodies that have an interior point in common. However, $\Delta$ may not be continuous in general; continuity  fails if there are two convergent sequences of convex bodies $K_i \to K$ and $L_i\to L$ such that $K_i\cap L_i=\varnothing$, $K\cap L\neq \varnothing$ and $V(K\cap L)>0$.

\section{Asymptotic estimates}\label{sec:asymptotic_estimates}

Recall that  $|D_n| = \pi^{\frac{n}{2}}/ \Gamma(\frac{n}{2}+1)$ and $|D_n| = \frac{1}{n} |\partial D_n| = \frac{2\pi}{n} |D_{n-2}|$.
By Stirling's inequality,
\begin{equation}\label{eqn:gamma_general}
	\sqrt{2\pi x} \left(\frac{x}{e}\right)^{x} 
	\leq \Gamma(x+1) 
	\leq \sqrt{2\pi x} \left(\frac{x}{e}\right)^{x} e^{\frac{1}{12x}}
	\leq \sqrt{2\pi x} \left(\frac{x}{e}\right)^{x} \left(1+\frac{1}{x}\right),
	\quad \forall x\geq 1.
\end{equation}
This implies
\begin{equation}\label{eqn:vol_Dn}
	\frac{1}{\sqrt{\pi n}} \left(\frac{2\pi e}{n}\right)^{\frac{n}{2}} \left(1-\frac{1}{n}\right)
	\leq |D_n| 
	\leq \frac{1}{\sqrt{\pi n}} \left(\frac{2\pi e}{n}\right)^{\frac{n}{2}}, \quad \forall n\geq 1
\end{equation}
and
\begin{equation}\label{eqn:vol_Dn_and_partial_Dn}
	\sqrt{2\pi n} \left(1-\frac{1}{n}\right) \leq V_1(D_n)= \frac{n|D_n|}{|D_{n-1}|} \leq \sqrt{2\pi n}, \qquad \forall n\geq 2.
\end{equation}
Therefore, we also obtain
\begin{equation}\label{eqn:partial_Dn}
	\frac{2\pi e}{n}
	\leq |\partial D_n|^{\frac{2}{n-1}} 
	\leq \frac{2\pi e}{n} (2e)^{\frac{1}{n-1}} 
	\leq \frac{2\pi e}{n} \left(1+\frac{8}{n}\right), \qquad \forall n\geq 2,
\end{equation}
which proves \eqref{eqn:asymptotic_Dn}.
We use the local approximation $1-\frac{x}{2} \leq \Gamma(1+x) \leq 1+ x$ for $x\in[0,2]$, which yields
\begin{equation}\label{eqn:gamma_small}
	1-\frac{2}{n} \leq \Gamma\left(1+\frac{2}{n-1}\right) \leq 1+\frac{2}{n}, \qquad \forall n\geq 2.
\end{equation}

\paragraph{Estimates for \eqref{eqn:random_constants} and \eqref{eqn:beta}.}

We first estimate the constant in \eqref{eqn:random_constants}, which is defined by
\begin{equation*}
	\alpha(n,j) = \left(1-\frac{2}{n+1}\right) \left(\frac{n|D_n|}{|D_{n-1}|}\right)^{\frac{2}{n-1}} \frac{\Gamma\left(j+1+\frac{2}{n-1}\right)}{\Gamma(j+1)}.
\end{equation*}
By the formula $\Gamma(z+1) = z\Gamma(z)$ for $z>0$, we obtain
\begin{equation}\label{eqn:increasing_est}
	\frac{\Gamma\left(j+1 + \frac{2}{n-1}\right) }{\Gamma(j+1)}
	= \Gamma\left(1+\frac{2}{n-1}\right)\prod_{k=1}^j \left(1+\frac{2}{k(n-1)}\right)
	\leq \left(1+\frac{2}{n}\right) \exp\left( \frac{2}{n-1} \sum_{k=1}^j \frac{1}{k}\right).
\end{equation}
This implies
\begin{equation}\label{eqn:j_n_estimate}
	1 \leq \frac{\Gamma\left(j+1+\frac{2}{n-1}\right)}{\Gamma(j+1)} 
	\leq 1+ 25 \frac{\ln(j+1)}{n}, \qquad \forall  n\geq 2, \, \forall j\in[n] 
\end{equation}
and 
\begin{equation*}
	\frac{\Gamma\left(n+1+\frac{2}{n-1}\right)}{\Gamma(n+1)} 
	\leq 1+ 4 \frac{\ln n}{n}, \qquad \forall n\geq 10.
\end{equation*}
Using this estimate as well as \eqref{eqn:gamma_general}, \eqref{eqn:partial_Dn} and \eqref{eqn:gamma_small}, we conclude
\begin{equation}\label{eqn:alpha_estimate}
	1+\frac{\ln n}{n} - \frac{2}{n}\leq \alpha(n,j) \leq 1+ 120 \frac{\ln n}{n}, \qquad \forall  n\geq 2, \, \forall j\in[n], 
\end{equation}
and $\alpha(n,j) \geq 1$ for all $n\geq 4$ and $j\in[n]$.
Furthermore, by \eqref{eqn:gamma_general} we conclude that for all $n\geq 2$ and $j\in[n-1]$,
\begin{equation*}
	\frac{\alpha(n,n)}{\alpha(n,j)} 
	= \frac{\Gamma(n+1+\frac{2}{n-1})}{\Gamma(j+1+\frac{2}{n-1})} \frac{\Gamma(j+1)}{\Gamma(n+1)} 
	\leq 1+ \frac{1}{j}
\end{equation*}
and
\begin{equation*}
	\frac{\alpha(n,n)}{\alpha(n,j)} = \prod_{k=j+1}^n \left(1+\frac{2}{k(n-1)}\right)\geq 1+\frac{2}{n^2}.
\end{equation*}
Moreover, we obtain the estimate
\begin{equation*}
	\frac{\alpha(n,n)}{\alpha(n,j)} \leq \frac{\alpha(n,n)}{\alpha(n,1)} \leq 1+ 3 \frac{\ln n}{n}, \qquad \forall n\geq 2, \; \forall j\in[n-1],
\end{equation*}
which yields
\begin{equation}\label{eqn:alpha_rel_est}
	1+\frac{2}{n^2} \leq \frac{\alpha(n,n)}{\alpha(n,j)} \leq 1 + \min \left\{ \frac{1}{j} , 3 \frac{\ln n}{n} \right\}, \qquad \forall n\geq 2, \; \forall j\in[n-1].
\end{equation}
For \eqref{eqn:beta}, we  write
	$\beta(n,j) = \alpha(n,j) \frac{j V_j(D_n)}{2n|D_n|} |\partial D_n|^{-\frac{2}{n-1}}$
and use \eqref{eqn:partial_Dn} and \eqref{eqn:alpha_estimate} to derive the bounds
\begin{equation}\label{eqn:beta_estimate}
	\frac{j V_j(D_n)}{4\pi e|D_n|} \left(1+14 \frac{\ln n}{n}\right) \leq \beta(n,j)\leq \frac{j V_j(D_n)}{4\pi e|D_n|} \left(1+120\frac{\ln n}{n}\right),
	\qquad \forall n\geq 2, \; \forall j\in[n].
\end{equation}

\paragraph{Estimates for $\del_{n-1}$ which yield \eqref{eqn:random_approx_dual} and \eqref{eqn:asymtotics_random_weighted}.}

In \cite[Thm.\ 2]{MankiewiczSchutt:2001} it was shown that
\begin{equation*}
	\frac{n-1}{n+1} |D_{n-1}|^{-\frac{2}{n-1}} 
	\leq \del_{n-1} 
	\leq \frac{n-1}{n+1} |D_{n-1}|^{-\frac{2}{n-1}} \frac{\Gamma\left(n+1+\frac{2}{n-1}\right)}{\Gamma(n+1)}, \qquad \forall n\geq 2.
\end{equation*}
Using \eqref{eqn:vol_Dn_and_partial_Dn}, \eqref{eqn:partial_Dn} and \eqref{eqn:j_n_estimate}, this yields 
\begin{equation}\label{eqn:del_estim}
	\frac{n}{2\pi e} \left(1+\frac{\ln n}{n} - \frac{2}{n}\right)
	\leq \del_{n-1} 
	\leq \frac{n}{2\pi e} \left(1+25 \frac{\ln n}{n}\right), \qquad \forall n\geq 2
\end{equation}
and
\begin{equation*}
	\frac{n}{2\pi e} \left(1+ \frac{1}{8} \frac{\ln n}{n}\right) \leq \del_{n-1} \leq \frac{n}{2\pi e} \left(1+4 \frac{\ln n}{n}\right),
	\qquad \forall n\geq 10.
\end{equation*}
Using \eqref{eqn:beta_estimate} we conclude
\begin{equation}\label{eqn:beta_estim_ok}
	1 + 8 \frac{\ln n}{n} 
	\leq\frac{2n^{\frac{2}{n-1}} \beta(n,n)}{\del_{n-1}}
	\leq 1+1000\frac{\ln n}{n}, 
	\qquad \forall n\geq 10,
\end{equation}
which yields \eqref{eqn:asymtotics_random_weighted}. Now \eqref{eqn:random_approx_dual} follows similarly from
\begin{equation}\label{eqn:random_approx_dual_estim}
	1+3\frac{\ln n}{n} \leq \lim_{N\to\infty} \frac{\E \widetilde{\Delta}_{j}(K,P_N^{\widetilde{\psi}_j})}{\widetilde{\Delta}_{j}(K,\cP_N^i(K))} 
	= \frac{2\beta(n,n)}{\del_{n-1}} \leq 1+200 \frac{\ln n}{n}, \qquad \forall n\geq 10.
\end{equation}

\paragraph{Estimates for $\div_{n-1}$ which yield \eqref{eqn:del_div} and Corollary \ref{cor:AffentrangerRemark}.}
By \cite[Thm.\ 4]{HoehnerKur:2018}, there are absolute constants $C_1,C_2$ such that $C_2 > C_1 > 0$ and
\begin{equation*}
	\div_{n-1} \geq \frac{n}{2\pi e} \left(1+\frac{\ln n}{n} - \frac{C_1}{n}\right), \qquad \forall n\geq C_2.
\end{equation*}
Hence, by \eqref{eqn:del_div_ineq} and \eqref{eqn:del_estim} we derive
\begin{equation}\label{eqn:div_del_estim}
	1\leq \frac{\del_{n-1}}{\div_{n-1}} \leq 1 + (25 + C_3) \frac{\ln n}{n}, \qquad \forall n\geq C_2,
\end{equation}
where $C_3:=\frac{C_1C_2}{C_2-C_1} >0$. Thus, to prove Corollary \ref{cor:AffentrangerRemark} we use \eqref{eqn:asymptotic_bounds}, \eqref{eqn:affentrangerball1}, \eqref{eqn:partial_Dn} and \eqref{eqn:alpha_estimate} to obtain
\begin{equation}\label{eqn:random_approx_estim}
	1\leq \limsup_{N\to\infty} \frac{\E \Delta_j(D_n, P_N)}{\Delta_j(D_n,\cP_N^i)} 
	\leq \frac{\alpha(n,j)}{\div_{n-1}|\partial D_n|^{\frac{2}{n-1}}} \leq 1+(120+C_3) \frac{\ln n}{n}, \qquad \forall n\geq C_2.
\end{equation}
This also yields \eqref{eqn:asymtotics_random_weighted_circum} since
\begin{equation}\label{eqn:asymtotics_random_weighted_circum_estim}
	1\leq \lim_{N\to\infty} \frac{\E \widetilde{\Delta}_j(K, P_{(N)}^{\widetilde{\phi}_j}\cap L)}{\widetilde{\Delta}_j(K,\cP_{(N)}^o(K))} 
	= \frac{n^{\frac{2}{n-1}}\alpha(n,1)}{\div_{n-1} |\partial D_n|^{\frac{2}{n-1}}} 
	\leq 1+ (180+2C_3) \frac{\ln n}{n}, \qquad \forall n\geq C_2. 
\end{equation}


\section{\texorpdfstring{The relationship between $\Delta_1$ and $\delta_1$}{Intrinsic width approximation}}\label{sec:width_rel}

The intrinsic width $V_1(K) = V_1(D_n)\int_{\Sp} h_K(u)\, d\sigma(u)$ is  related to the $L^1$ metric $\delta_1(K,L) = \int_{\Sp} |h_K(u)-h_L(u)|\, d\sigma(u)$ (see, e.g.\ \cite{Florian:1989}). The intrinsic width also defines the intrinsic width deviation $\Delta_1(K,L) = V_1(K)+V_1(L) - 2V_1(K\cap L)$, which is not a metric as we saw in Appendix \ref{sec:not_a_metric}.

\begin{theorem}\label{thm:delta_1}
	Let $K,L\in\cK(\R^n)$. Then
	\begin{equation*}
		\Delta_1(K,L) \geq V_1(D_n) \delta_1(K,L)
	\end{equation*}
	with equality if and only if $K\cup L$ is convex. In particular, $\Delta_1(K,L) = V_1(D_n) \delta_1(K,L)$ if $K\subset L$.
\end{theorem}
\begin{proof}
	Since $|a-b|+a+b=2\max\{a,b\}$ and $\max\{h_K,h_L\} = h_{\conv(K\cup L)}$, we derive
	\begin{align*}
		\delta_1(K,L) 
		&= \int_{\Sp} |h_K(u)-h_L(u)|\, d\sigma(u) \\
		&= \int_{\Sp} \left(2\max\{h_K(u),h_L(u)\} - h_K(u) - h_L(u)\right)\, d\sigma(u)\\
		&= \int_{\Sp}  2h_{\conv(K\cup L)}(u) \, d\sigma(u) - \int_{\Sp} h_K(u)\, d\sigma(u) - \int_{\Sp} h_L(u)\, d\sigma(u)\\
		&= \frac{1}{V_1(D_n)} \bigg(2V_1(\conv(K\cup L)) - V_1(K) - V_1(L)\bigg).
	\end{align*}
	We also have that
	\begin{equation*}
		\frac{V_1(K)+V_1(L)}{V_1(D_n)} 
		= \!\int_{\Sp}\! (h_K(u)+h_L(u))\, d\sigma(u) 
		= \!\int_{\Sp}\!(\max\{h_K(u), h_L(u)\} + \min\{h_K(u),h_L(u)\}) \, d\sigma(u).
	\end{equation*}
	Now since $\min\{h_K,h_L\} \geq h_{K\cap L}$, we derive $V_1(\conv(K\cup L)) \leq V_1(K)+V_1(L) - V_1(K\cap L)$. Therefore,
	\begin{equation*}
		V_1(D_n) \delta_1(K,L) \leq V_1(K) + V_1(K) - 2V_1(K\cap L) = \Delta_1(K,L)
	\end{equation*}
	with equality if and only if $V_1(K)+V_1(L) = V_1(\conv(K\cup L)) + V_1(K\cap L)$, or equivalently, if and only if $\min\{h_K,h_L\} = h_{K\cap L}$.
	If $K\cup L$ is convex, then by the valuation property of $V_1$ we have $V_1(K)+V_1(L) = V_1(K\cap L) + V_1(K\cup L)$.
	
	We are done once we show that $V_1(K)+V_1(L) = V_1(\conv(K\cup L)) + V_1(K\cap L)$ implies that $K\cup L$ is convex.
	Assume the opposite. Then $h_{K\cap L}(u) = \min\{h_K(u),h_L(u)\}$ for all $u\in \Sp$ and there exists a point $z\in \conv(K\cup L) \setminus (K\cup L)$.
	Since $z \not\in K$, there exists $u_1\in \Sp$ such that $ z\cdot u_1 > h_K(u_1)$. Since $z\in \conv(K\cup L)$, we also have $z\cdot u_1 \leq \max\{h_K(u_1),h_L(u_1)\}$. Hence,
	\begin{equation*}
		h_{K\cap L}(u_1) = h_K(u_1) < z\cdot u_1 \leq h_L(u_1).
	\end{equation*}
	Analogously, there exists $u_2 \in \Sp$ such that
	\begin{equation*}
		h_{K\cap L}(u_2) = h_L(u_2) < z\cdot u_2 \leq h_K(u_2).
	\end{equation*}
	Observe that $u_1\neq u_2$.
	If $u_1=-u_2$, then $h_K(u_1) < -h_L(-u_1)$, or equivalently,
	\begin{equation*}
		\max_{x\in K} x\cdot u_1 < \min_{y\in L}  y\cdot u_1.
	\end{equation*}
	Hence $K$ and $L$ can be strictly separated by a hyperplane with normal direction $u_1$, that is, $K\cap L =\varnothing$, which is a contradiction to $h_{K\cap L} = \min\{h_K,h_L\}$.
	
	Thus we may assume that $u_1\neq u_2$ and $u_1 \neq -u_2$, i.e., there is a unique minimizing geodesic arc between $u_1$ and $u_2$ on $\Sp$. By the continuity of $h_K(u)-h_L(u)$, there exists $u_3$ on this arc such that $h_K(u_3)=h_L(u_3)$. We may write $u_3=\alpha u_1+\beta u_2$ for some $\alpha,\beta>0$.
	By the subadditivity property of support functions, we conclude
	\begin{equation*}
		h_{K\cap L}(u_3) 
		\leq \alpha h_{K\cap L}(u_1)+\beta h_{K\cap L}(u_2) 
		<  z \cdot u_3  \leq \max\{h_K(u_3),h_L(u_3)\} = \min\{h_K(u_3),h_L(u_3)\},
	\end{equation*}
	which is also a contradiction to $h_{K\cap L} = \min\{h_K,h_L\}$. Hence $K\cup L$ is convex.
\end{proof}

\begin{remark}
    If $K,L\in\cK(\R^n)$ are such that $rD_n\subset K\cap L$ and $K\cup L\subset RD_n$ for some $r,R>0$, then by Theorem \ref{thm:delta_1} and $h_{K\cap L} \leq \min\{h_K,h_L\} \leq (R/r) h_{K\cap L}$, we derive
    \begin{equation*}
        V_1(D_n)\delta_1(K,L) \leq \Delta_1(K,L) \leq \frac{R}{r} V_1(D_n)\delta_1(K,L),
    \end{equation*}
In particular, this yields that approximations with respect to $\Delta_1$ and $\delta_1$ are of the same order.
\end{remark}

\begin{figure}[t]
	\centering
	\begin{tikzpicture}
        \clip (-8,-4.5) rectangle (8,3.5);
		\node at (0,0) {\includegraphics[width=0.8\textwidth]{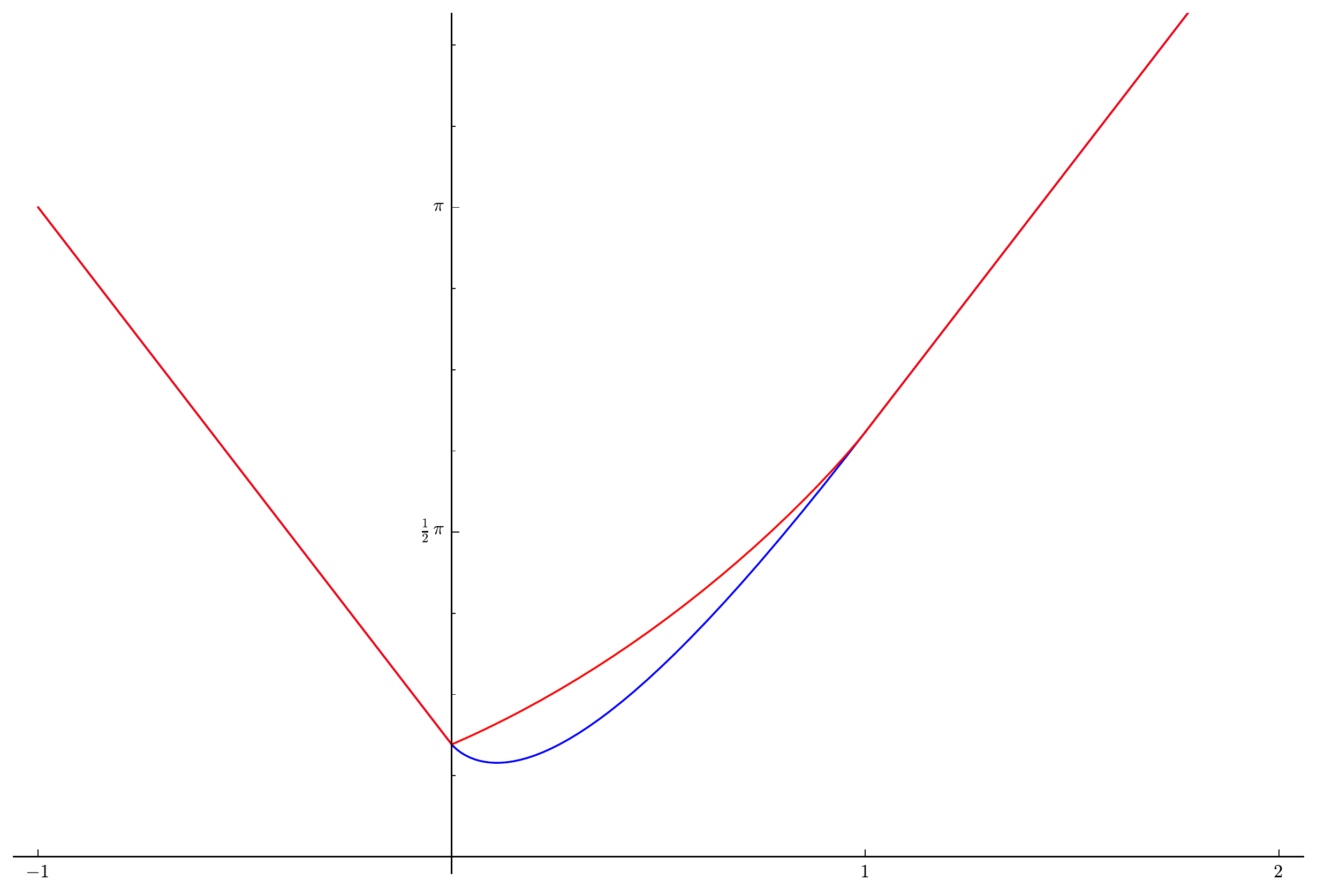}};
		\node[red, rotate=35] at  (-0.5,-1.8) {\small $\Delta_1(D_2,T(h))$};
		\node[blue] at (0,-3.5) {\small $\pi\delta_1(D_2,T(h))$};
		
		\begin{scope}[scale=1.7,xshift=2.2cm,yshift=-1cm]
			\def\h{0.4}
			\fill[blue, opacity=0.2] (0,0) circle(1);
			\draw (0,0) circle(1);
			\node at (45:1.2) {$D_2$};
			\fill[orange, opacity=0.2] (90:{1+\h}) -- (210:{1+\h}) -- (330:{1+\h}) -- cycle;
			\draw (90:{1+\h}) -- (210:{1+\h}) -- (330:{1+\h}) -- cycle;
			\node at (330:{1+\h+0.3}) {$T(h)$};
			\draw (90:0) -- (90:{1+\h}) node[midway,rotate=90, above] {\small $1+h$};
		\end{scope}

	\end{tikzpicture}
	\caption{Plot of the $L^1$ distance $\delta_1(D_2,T(h))$ and the first intrinsic volume deviation $\Delta_1(D_2,T(h))$ between the disk $D_2$ and a regular triangle $T(h)$ with circumradius $1+h$.}
	\label{fig:1}
\end{figure}
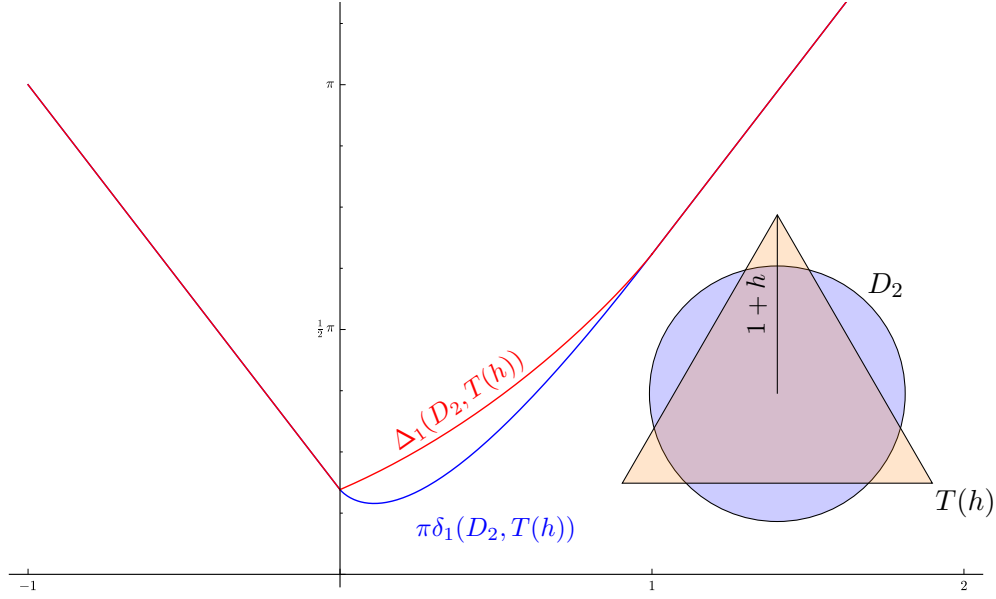

\begin{example}
	Consider the unit disk $D_2$ in $\R^2$. 
	For $h\in (-1, \infty)$ let $T(h)$ be a regular triangle centered at the origin with circumradius $1+h$, that is, $T(h)$ is inside of $D_2$ if $h<0$ and $D_2$ is inside of $T(h)$ if $h>1$. Thus, since $V_1(D_2)=\pi$ we derive
	\begin{equation*}
		\pi\delta_1(D_2,T(h)) = \Delta_1(D_2,T(h)) = \begin{cases} 
		                                                	\pi-\frac{3\sqrt{3}}{2}(1+h) & \text{if $h\in (-1,0]$},\\
		                                                	\frac{3\sqrt{3}}{2}(1+h)-\pi & \text{if $h\in [1, \infty)$}.
		                                                \end{cases}
	\end{equation*}
	For $h\in (0,1)$, we  calculate that
	\begin{align*}
		\pi\delta_1(D_2,T(h)) = -2\pi- \frac{3\sqrt{3}}{2}(1+h) + 6\sqrt{2h+h^2}+6\arcsin\left(\frac{1}{1+h}\right)
	\intertext{and}
		\Delta_1(D_2,T(h)) =\pi + \frac{3\sqrt{3}}{2}(1+h) - \sqrt{3}\sqrt{9-6h-3h^2} - 6\arccos\left(\frac{1+h+\sqrt{9-6h-3h^2}}{4}\right).
	\end{align*}
	See Figure \ref{fig:1} for a plot of the two functions. In particular,  the minimum of $\delta_1$ is achieved for some $h\in(0,1)$, i.e., $D_2$ and $T(h)$ are in a general position, whereas the minimum of $\Delta_1$ is achieved for $h=0$, that is, if the regular triangle is inscribed. 
	The latter also follows as a special case of a theorem by Eggleston \cite[Lem.\ 4]{Eggleston:1957}, who showed that the best-approximating polygon with respect to $\Delta_1$ is always inscribed. Note that in the plane $\R^2$, the first intrinsic volume deviation $\Delta_1$ is the perimeter deviation; see, e.g., \cite{Fodor:2019}.
\end{example}


\pagebreak
\small
\vspace{2mm}

\noindent {\sc Institute of Discrete Mathematics and Geometry, Vienna University of Technology, Wiedner Hauptstrasse 8--10, 1040 Vienna, Austria}

\noindent {\it E-mail address:} {\tt florian.besau@tuwien.ac.at}

\vspace{2mm}

\noindent {\sc Department of Mathematics \& Computer Science, Longwood University, 201 High St, Farmville, VA 23909}

\noindent {\it E-mail address:} {\tt hoehnersd@longwood.edu}

\vspace{2mm}

\noindent {\sc Massachusetts Institute of Technology, Computer Science \& Artificial Intelligence Laboratory, 32 Vassar St, Cambridge, MA 02139}

\noindent {\it E-mail address:} {\tt gilkur@mit.edu}

\end{document}